\documentclass{amsart}
\usepackage[foot]{amsaddr}
\usepackage[margin=1.5in]{geometry}
\usepackage{amsthm,amssymb,amsfonts,amsmath}
\usepackage{stmaryrd}
\usepackage{bbm} %For indicator 1
\usepackage[numbers, sort&compress]{natbib}
\usepackage{graphicx}
\usepackage{color}
\usepackage{enumerate}
\usepackage{mathtools}
\usepackage{ulem}
\usepackage{enumitem}
\usepackage{tikz}
\usetikzlibrary{arrows.meta, positioning}
\usepackage{hyperref}
\usepackage[noabbrev,capitalise]{cleveref}

\providecommand{\R}{}

\renewcommand{\R}{\mathbb{R}}

\newcommand{\set}[1]{\left\{ #1 \right\}}
\newcommand{\Esub}[2]{{\mathbf E_{#1}}\left[#2\right]}

\newcommand\cB{\mathcal B}

\newcommand\cF{\mathcal F}

\newcommand\cH{\mathcal H}

\newcommand\cP{\mathcal P}

\newcommand\cU{{\mathcal U}}

\newcommand{\bE}{\mathbf{E}}

\newcommand{\bP}{\mathbf{P}} 
\newcommand{\bQ}{\mathbf{Q}} 
\newcommand{\bR}{\mathbf{R}}

\newcommand{\f}{\mathfrak{f}} 
\newcommand{\g}{\mathfrak{g}} 

\renewcommand{\t}{\mathbf{\tau}}
\renewcommand{\c}{\mathfrak{c}}   % cemetery

\newcommand{\asymbol}[1]{\ooalign{$#1\circledcirc$\cr\hidewidth$#1\bullet$\hidewidth}}
\renewcommand{\a}{\mathbin{\mathpalette\asymbol{}}}
\renewcommand{\d}{\otimes}   % dead regime (lambda 2)
\newcommand{\dbleft}{[\![}
\newcommand{\dbright}{]\!]}
\newcommand{\db}[1]{\dbleft #1 \dbright}

\newcommand{\1}{\textup{(1)}} % Scenario 1 
\newcommand{\2}{\textup{(2)}}   % Scenario 2 
\newcommand{\3}{\textup{(3)}}   % Scenario 3 
\newcommand{\4}{\textup{(4)}}    % Scenario 4 

\newcommand{\rSet}{\set{\a,\d}} %symbol for the regime set
\newcommand{\pran}[1]{\left(#1\right)}
\newcommand{\brac}[1]{\left[#1\right]}
\newcommand{\inn}[1]{\langle#1\rangle}
\newcommand{\abs}[1]{\left|#1\right|}

\newcommand{\KL}[2]{\mathrm{KL}(#1||#2)}
\providecommand{\eps}{}
\renewcommand{\eps}{\varepsilon}
\providecommand{\ora}[1]{}
\renewcommand{\ora}[1]{\overrightarrow{#1}}

\newcommand\one{\mathbbm{1}}
\newcommand\proc{\set{(X_t,\Lambda_t)}_{[0,T]}}
\newcommand\aug[1]{\db{#1,0}}
\newcommand\sce[1]{\textup{Scenario (#1)}}
\newcommand\kt{\mathfrak{K}} % killing time
\makeatletter
\newcommand{\customlabel}[2]{%
   \protected@write \@auxout {}{\string \newlabel {#1}{{#2}{\thepage}{#2}{#1}{}} }%
   \hypertarget{#1}{}
}
\makeatother

\newtheorem{thm}{Theorem}
\newtheorem{lem}[thm]{Lemma}
\newtheorem{prop}[thm]{Proposition}
\newtheorem{cor}[thm]{Corollary}
\newtheorem{dfn}[thm]{Definition}
\AddToHook{env/prop/begin}{\crefalias{thm}{prop}}
\AddToHook{env/lem/begin}{\crefalias{thm}{lem}}
\AddToHook{env/cor/begin}{\crefalias{thm}{cor}}
\AddToHook{env/dfn/begin}{\crefalias{thm}{dfn}}

\numberwithin{thm}{subsection}
\numberwithin{equation}{subsection}
\definecolor{clc}{rgb}{0,0.53,0.74}
\definecolor{caz}{rgb}{0,0.5,0}

\usepackage{tikz}
\usepackage{tikz-3dplot}
\tdplotsetmaincoords{60}{30}
\begin{document}
\begin{abstract}
    The unbalanced Schr\"odinger bridge problem (uSBP) seeks to interpolate between a probability measure $\rho_0$ and a sub-probability measure $\rho_T$ while minimizing KL divergence to a reference measure $\bR$ on a path space. In this work, we investigate the case where $\bR$ is the path measure of a diffusion process with killing, which we interpret as a regime-switching diffusion. In addition to matching the initial and terminal distributions of trajectories that survive up to time $T$, we consider a general constraint $\psi(t,x)$ on the distribution of killing times and/or killing locations. 
    
    We investigate the uSBPs corresponding to four choices of $\psi$ in detail which reflect different levels of information available to an observer. We also provide a rigorous analysis of the connections and the comparisons among the outcomes of these four cases. Our results are novel in the field of uSBP. The regime-switching approach proposed in this work provides a unified framework for tackling different uSBP scenarios, which not only reconciles but also extends the existing literature on uSBP.
\end{abstract}
\title{A Regime-Switching Approach to the Unbalanced Schr\"odinger Bridge Problem}
\author{Andrei Zlotchevski$^{1,\dagger}$}
\author{Linan Chen$^{1,\dagger,\ast}$}
\address{$^{1}$Department of Mathematics and Statistics, McGill University, 
Montreal, Canada}
\address{$^{\ast}$Corresponding author}
\address{$^{\dagger}$Equal contributors}
\thanks{We acknowledge the support of the Natural Sciences and Engineering Research Council of Canada (NSERC), awards 559387-2021 and 241023.}
\email{andrei.zlotchevski@mail.mcgill.ca}

\email{linan.chen@mcgill.ca}

\maketitle

\section{Introduction}
This paper is dedicated to the unbalanced Schr\"odinger bridge problem (uSBP) for diffusion processes. Roughly speaking, given probability distributions $\rho_0$ and $\rho_T$ the classic Schr\"odinger bridge problem (SBP) aims to ``bridge'' the two distributions in a way that minimizes KL divergence (relative entropy) to a reference path measure $\bR$. The unbalanced Schr\"odinger bridge problem, on the other hand, considers the case where $\rho_T$ is a sub-probability measure (i.e. with total mass less than 1). In this Introduction, we shall describe the type of reference measure under consideration and then rigorously formulate the uSBP as a classic SBP with regime switching.

There are two main objectives to this work. First, we seek to establish a fundamental theory of the diffusion uSBP that can model a broad range of constraints imposed on the distribution of killed paths. By applying the theory of the SBP for (jump) diffusions with regime switching \cite{zlotchevski2025schrodinger}, we will thoroughly investigate the uSBP for diffusions and derive many of its properties in full generality. Then, we aim to connect the existing uSBP literature \cite{chen2022most,eldesoukey2024excursion,eldesoukey2025inferring} under the regime-switching framework adopted in this work and demonstrate how these different problems are directly related. In fact, our approach will allow us to recover the existing results from other uSBP research and to offer new insight on some aspects of the uSBP.

\subsection{Set-up of Our Model}
Our model is set on a ``hybrid'' state space $\R^d\times \rSet$ over a finite time interval $t\in[0,T]$. Here, $\a$ and $\d$ are called the \textit{regimes} of the process, where $\a$ is the ``active'' regime and $\d$ is the ``dead'' regime. Let $\Omega$ be a certain class, which will be specifically defined later, of c\`adl\`ag paths $\omega : t\in[0,T] \mapsto(\omega^X(t),\omega^{\Lambda}(t))\in \R^d\times\rSet$, equipped with the Skorokhod metric. For every $t\in[0,T]$, define $(X_t,\Lambda_t):\Omega\rightarrow\R^d\times\rSet$ to be the time projection such that, for every $\omega\in\Omega$, $X_t(\omega)=\omega^X(t)$ gives the spatial position of the path, and $\Lambda_t(\omega)=\omega^\Lambda(t)$ is the regime of the path. If $\cF:=\cB_\Omega$ is the Borel $\sigma$-algebra on $\Omega$, then $\cF$ is generated by the time projections $\set{(X_t,\Lambda_t)}_{[0,T]}$. Let $\set{\cF_t}_{[0,T]}$ be the natural filtration (or its usual augmentation) associated with $\set{(X_t,\Lambda_t)}_{[0,T]}$. Denote by $\cP(\Omega)$ the set of probability measures on $(\Omega,\cF)$. For $\bP\in\cP(\Omega)$, to which we refer as a \textit{path measure}, we say that $\set{(X_t,\Lambda_t)}_{[0,T]}(\omega)$ is the \textit{canonical process} under $\bP$ if $\omega\in\Omega$ is sampled under $\bP$.  Finally, for $0\leq t \leq T$, we write the distribution of $(X_t,\Lambda_t)$ under $\bP$ as $\bP_t$, and call it the \textit{marginal distribution} at time $t$; similarly, the joint distribution of $(X_0,\Lambda_0)$ and $(X_T,\Lambda_T)$ under $\bP$ is denoted by $\bP_{0T}$.\\

We consider a path measure $\bR\in\cP(\Omega)$, called the \textit{reference measure}, such that the canonical process $\set{(X_t,\Lambda_t)}_{[0,T]}$ under $\bR$ is a \textit{regime-switching diffusion} in the sense we will describe below. 

First, the regime component $\set{\Lambda_t}_{[0,T]}$ is a discrete jump process on the state space $\set{\a,\d}$.
We assume that all particles start in the active regime, meaning $\Lambda_0 = \a$, and that they are ``killed'' (i.e., switch to regime $\d$) at a certain rate depending on their position and the current time. More specifically, for $t\in[0,T)$,
\begin{equation}\label{eqn: killing mechanism}
    \lim_{\eps \downarrow 0}\frac{1}{\eps}\bR(\Lambda_{t+\eps}= \d | \Lambda_t = \a, X_t=x) = V(t,x),
\end{equation}
where $V:[0,T]\times\R^d\to\R_+$ is called the  \textit{killing rate} under $\bR$.
Moreover, $$\bR(\Lambda_{t'}=\a | \Lambda_t=\d)=0\;\textup{ for all }0\leq t\leq t'\leq T,$$meaning that dead particles are not allowed to ``revive'' and switch back to the active regime. Define $\kt:=\inf\set{t\geq 0:\Lambda_t=\d}$ to be the \textit{killing time}. We observe that, as a random variable on $\Omega$, $\kt$ is a stopping time since $\{\kt>t\}=\{\Lambda_t=\a\}$ for every $t\in[0,T]$.

Next, while the particle is in the active regime, 
the spatial component $\set{X_t}_{[0,T]}$ under $\bR$ is governed by a stochastic differential equation (SDE) as
\begin{equation}\label{SDE for Xt}
 \begin{aligned}   X_t&=X_0+ \int_0^t b(s,X_{s})ds +\int_0^t\sigma(s,X_{s})dB_s \;\textup{ for }0\leq t<\kt\,\wedge T,
\end{aligned}
\end{equation}
where $\set{B_s}_{s\geq 0}$ is standard $d$-dimensional Brownian motion and $\kt$ is the killing time defined above. We discuss conditions on the coefficients of the SDE later. 

In addition, if the regime switches from $\a$ to $\d$ at time $\kt$, then the particle jumps in the spatial coordinate from $X_{\kt-}$ to a new position $\psi(\kt,X_{\kt-})$. 
In the dead regime, the drift and diffusion coefficients are set to $0$, and therefore the particle is frozen at $\psi$.   Below is the visualization of a sample path killed at time $\kt$:
\newpage

\begin{figure}[h!]
\centering
\begin{tikzpicture}[x=1cm,y=1cm,z=1cm,tdplot_main_coords,
    scale=1.1, every node/.style={font=\small}]

% Draw two parallel planes 

% Plane 1 at z=0, drawn in xy plane
\begin{scope}[canvas is xy plane at z=0]
    \draw[thick, gray] (-0.5, -0.5) -- (4.5, -0.5) -- (4.5, 4.5) -- (-0.5, 4.5) -- cycle;
    \node[below right] at (4.5, 4.5) {Regime $\a$ (active particle)};
\end{scope}

% Plane 2 at z=2, drawn in xy plane
\begin{scope}[canvas is xy plane at z=2]
    \draw[thick, gray] (-0.5, -0.5) -- (4.5, -0.5) -- (4.5, 4.5) -- (-0.5, 4.5) -- cycle;
    \node[above right] at (4.5, 4.5) {Regime $\d$ (killed particle)};
\end{scope}

% Continuous curve in Plane 1
\draw[blue, thick, domain=0:4, samples=20, smooth] plot[variable=\x] ({\x}, {0.3*sin(\x r)}, 0);

% Jump (dashed line) to another point in Plane 2
\draw[red, dashed, thick] (4, {0.3*sin(4 r)}, 0) -- (1.5, 2.2, 2);

% Mark start point
\fill[blue] (0, 0, 0) circle (2pt) node[above right] {$X_0$};

% Mark killing point
\fill[black] (4, {0.3*sin(4 r)}, 0) circle (2pt) node[below left] {$X_{\kt-}$};

% Mark the ending point in killed regime
\fill[red] (1.5, 2.2, 2) circle (2pt) node[above right] {End point $\psi(\kt,X_{\kt-})$};

\end{tikzpicture}
\caption{Sample path killed at time $\kt$.}
\end{figure}
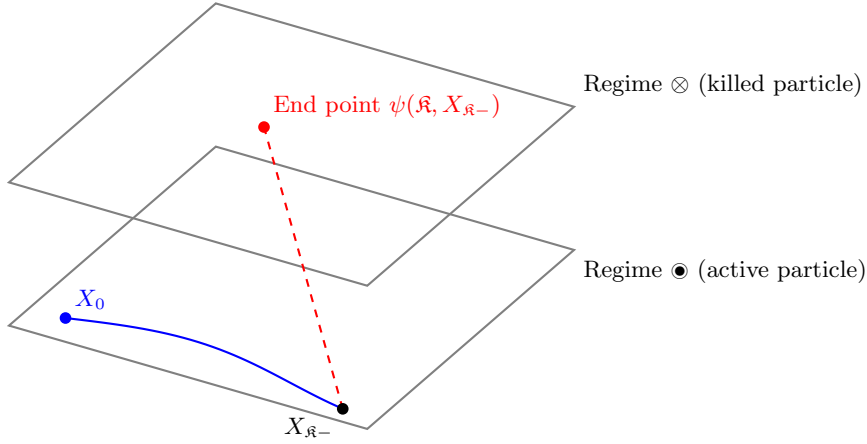
As a result, $\set{(X_t,\Lambda_t)}_{[0,T]}$ can be viewed a \textit{diffusion with killing}, where the killing action is interpreted as a regime switch $\a \to \d$. The measure $\bR_T(\cdot,\a)$, which is the distribution of particles that survive up to time $T$, is a sub-probability measure.\\

Given the dynamics of our SDE, we can choose an appropriate trajectory space $\Omega$.
\begin{dfn}\label{dfn:omega}
    Define $\Omega=\Omega(\psi)$ to be the collection of c\`adl\`ag functions $\omega=(\omega^X,\omega^\Lambda):[0,T]\to\R^d\times\rSet$ such that  
\begin{itemize}
    \item either, $\omega^\Lambda\equiv\a$ on $[0,T]$, in which case $\omega^X$ is continuous on $[0,T]$. 
    \item or, there exists a time $\kt=\kt(\omega)\in[0,T]$ such that $\omega^\Lambda\equiv\a$ on $[0,\kt)$ and $\omega^\Lambda\equiv\d$ on $[\kt,T]$, in which case $\omega^X$ is continuous on $[0,\kt)$ and $\omega^X\equiv\psi(\kt,\omega^X(\kt-))$ on $[\kt,T]$.
\end{itemize}
\end{dfn} 
We already know the time $\kt$ in the second scenario above as the killing time. We further identify the first scenario with the case ``$\kt>T$'', i.e.,  killing only occurs after time $T$. \\

In this paper, we study the following SBP.
\begin{dfn}[Unbalanced Schr\"odinger Bridge Problem]
    Given a reference measure $\bR\in\cP(\Omega,\cF)$ and two probability measures on $\R^d \times \set{\a,\d}$
    \[\rho_0:=\bigg(\rho_0(\cdot,\a),0\bigg) \text{ and }\rho_T:=\bigg(\rho_T(\cdot,\a),\rho_T(\cdot,\d)\bigg),
    \]
    the unbalanced Schr\"odinger bridge problem aims to find the measure $\widehat\bP\in\cP(\Omega)$ such that
\begin{equation}\label{dfn: SBP-dynamic}
    \widehat{\bP}=\arg\min\set{\KL{\bP}{\bR} \text{ such that } \bP_0=\rho_0,\bP_T=\rho_T }.
\end{equation}
\end{dfn}
Here, $\KL{\cdot}{\cdot}$ denotes the KL divergence (relative entropy), defined as $\KL{\bQ}{\bP}:=  \Esub{\bQ}{\log\pran{\frac{d\bQ}{d\bP}}}$ if $\bQ\ll \bP$, and $+\infty$ otherwise. The reference measure $\bR$ is such that the canonical process under $\bR$ is a diffusion with killing as described above. We say that a path measure $\bP$ is \textit{admissible} for this SBP if $\bP_0=\rho_0$ and $\bP_T=\rho_T$.
The solution $\widehat\bP$, if it exists, is called the \textit{Schr\"odinger bridge}. We will also call it the \textit{unbalanced Schr\"odinger bridge} because $\rho_0(\cdot,\a)$ is a probability measure, but $\rho_T(\cdot,\a)$ is a sub-probability measure. Since $\rho_0$ is only defined in the active regime, we write $\rho_0(\cdot,\a)$ and $\rho_0(\cdot)$ interchangeably.
\subsection{Background and Motivation}\label{Subsection: Scenarios description}
In the classic SBP, a large number of particles is observed at times $t=0$ and $t=T$ to have spatial distributions $\rho_0$ and $\rho_T$ respectively. One must then find the ``most probable way'' that the particles have evolved between the two observations, which is understood as finding the path measure $\widehat\bP$ that minimizes the KL divergence to the reference measure (prior belief) $\bR$. In particular, the SBP for diffusion processes is a long-celebrated problem dating back to E. Schr\"odinger's papers \cite{schrodinger1931uber,schrodinger1932theorie} in the 1930s and it has generated a rich literature; a selection of works include \cite{beurling1960automorphism,jamison1974reciprocal,jamison1975markov,follmer1988random,follmer1997entropy,daipra1991stochastic,pra1990markov,chen2016entropic}.
In the unbalanced SBP, the number of particles observed at time $T$ is expected to be less than that at time $0$. Therefore, one must reconstruct the ``most likely paths'' while accounting for the disappearance, or the killing, of the missing particles. The original work \cite{chen2022most} on uSBP treated this problem by considering an augmentation of the state space to include an absorbing ``cemetery state'' $\c$. The killing of a particle was interpreted as a jump to $\c$, and the problem was recast as a classic SBP over c\`adl\`ag paths in $\R^d \cup \set{\c}$. In this paper, we instead adopt a \textit{regime-switching} approach to this problem, following the example in \cite{zlotchevski2025schrodinger}. 
The advantage of our approach is two-fold: first, we can directly apply the general theory of the SBP for regime-switching (jump) diffusions developed in \cite{zlotchevski2025schrodinger} to study the uSBP, and second, we have additional flexibility in the use of the killing time $\kt$ and the spatial coordinate $X_{\kt-}$ for a killed particle, allowing us to model a variety of real-life situations with different types of observational data.  We elaborate more on the latter point below.
 
Although, as in the classic SBP framework, we are only allowed to make two observations on the path: first at time $0$, and then at time $T$, in our regime-switching setup, we are able to incorporate more information about the path than just what the endpoints can tell seemingly. This is due to the fact that any killed particle is frozen in ``time and space''. For example, if $\psi(t,x)=x$, then the spatial component $X_T$ not only gives the position at time $T$ of a particle that \textit{survived}, but also records the killing location of a dead particle. As a result, the terminal constraint $\rho_T(\cdot,\d)$ exactly prescribes the distribution of killing locations. 

Let us also highlight the pair of papers by Eldesoukey \textit{et al.} \cite{eldesoukey2024excursion,eldesoukey2025inferring} that treat the cases when certain target distributions are imposed on the killed particles. In contrast, the uSBPs we treat in this paper also include target distributions on the  particles that survived. We will draw connections to the above mentioned uSBP works in \cref{Section: connection to other paper}.

\subsection{The Four Scenarios}\label{subsection:four scenarios} 
In this work, we focus on four uSBP models that are highly relevant in applications, and we will examine various aspects of the Schr\"odinger bridge obtained for these models on a case-by-case basis.  

To begin, in all of our uSBP models, we assume that the observer knows the initial distribution $\rho_0(\cdot,\a)$, as well as the distribution of particles that survived up to time $T$, which is given by $\rho_T(\cdot,\a)$.
As for the killed particles, we investigate four different choices of $\psi$, to which we refer as \textit{Scenarios} of the uSBP. These choices reflect the information available to the observer regarding the killed particles:
\begin{enumerate}
    \item The observer knows the joint distribution of killing locations and killing times.
    \item The observer knows the distribution of killing times only.
    \item The observer knows the distribution of killing locations only.
    \item The observer knows the total mass of killed particles, with no knowledge of its distribution.
\end{enumerate}
As alluded to above, the types of constraints on the distribution of killed particles concerned in Scenarios (1), (2) and (3) have appeared in the uSBP literature \cite{eldesoukey2024excursion,eldesoukey2025inferring}, but these works did not impose a target distribution $\rho_T(\cdot,\a)$ for the surviving particles. The regime-switching approach will allow us to solve these uSBP scenarios in full generality. Scenario (4) is the original uSBP \cite{chen2022most}; we will fully recover known results and offer additional insight as well. Let us now explain our choice of $\psi$, as well as the dead-regime target distribution $\rho_T(\cdot,\d)$ for each of these Scenarios. We will add the superscript ``$\,^{(k)}$'', for $k\in\set{1,2,3,4}$, to indicate the Scenario to which these subjects correspond.

In Scenario (1), the observer is ``omniscient'' regarding the killed particles; that is, the observation provides the full joint distribution of killing locations and killing times. As we wish to record and match this $(d+1)$-dimensional information, we need to augment the dimension of the process. Let $\psi^{(1)}(t,x)=\db{x,t}$ for $ x\in \R^{d}$ and $t\in[0,T]$, and consider the augmentation of $\set{X_t}_{[0,T]}$ to a $\R^{d}\times[0,T]$-valued process $\set{\db{X_t,\t_t}}_{[0,T]}$ with $\t_t:=\kt\one_{[\kt,T]}(t)$ where, again, $\kt$ is the killing time of the process. This added last coordinate $\t_t$ will only be used in the dead regime to record the killing time. The augmented process is still within the setting of a diffusion with killing. Since $\t_t\equiv0$ for $t$ up to the killing time, we will omit the augmented coordinate and continue writing $(X_t,\Lambda_t)=(x,\a)$ in the active regime. In the dead regime, we will write the value taken by the augmented process as $(\db{X_t,\t_t},\Lambda_t)=(\db{x,\t},\d)$. The constraint $\rho_T^{(1)}(\db{x,\t},\d)$ is a sub-probability density on $\R^d\times[0,T]$ that encodes the observed joint distribution of killing locations and killing times, satisfying that  
\[
    \int_{\R^d}\rho_T(x,\a)dx+\int_0^T\int_{\R^d}\rho_T^\1(\db{x,\t},\d)dxd\t=1.
\]

In Scenario (2), we suppose that the observer keeps track of how many particles are active at any given time; in other words, the observation contains the distribution of killing times. To model this scenario, let $\c \in \R^d$ be an arbitrary fixed ``cemetery point'', and set $\psi^{(2)}(t,x)=\db{\c,t}$. Here, although the dimension augmentation is not required ($\c$ could be taken as a point in $\R^{d-1}$), the parallelism of the dimension size will simplify notation for the comparison in \cref{Subsection: Comparison}. A killed particle's terminal state will take the form $(\db{X_T,\t_T},\Lambda_T)=(\db{\c,\t},\d)$, and thus it will contain the killing time in the $\t$ coordinate, but not the killing location. Therefore, we can model the desired uSBP using the regime-switching approach by encoding the observer's knowledge of the killed particles into a target sub-probability density $\rho_T^{(2)}(\db{\c,\t},\d)$ on $[0,T]$ such that 
\[
    \int_{\R^d}\rho_T(x,\a)dx+\int_0^T\rho_T^\2(\db{\c,\t},\d)d\t=1.
\]

In Scenario (3), we consider the ``space only'' counterpart of Scenario (2), in the sense that the observer has knowledge of the distribution of killing locations, but not killing times. This scenario was studied for the jump-diffusion uSBP in \cite{zlotchevski2025schrodinger}. To model this problem, we let $\psi^{(3)}(t,x)=x$ and we encode the distribution of killing locations into the target $\rho_T^{(3)}(x,\d)$ as a sub-probability density on $\R^d$ with  
\[
    \int_{\R^d}\rho_T(x,\a)dx+\int_{\R^d}\rho_T^\3(x,\d)dx=1.
\]

Finally, in Scenario (4), we suppose that the observer has no knowledge about the distribution of killed particles and only knows their total amount. This is the set-up of the original uSBP \cite{chen2022most}. To model this scenario, take $\psi^{(4)}(t,x)=\c$, where $\c\in\R^d$ is a fixed cemetery point (as in Scenario (2)), and let $\rho_T^{(4)}(\c,\d)$ be a Dirac delta distribution representing the totality of ``missing'' particles; that is  
\[
    \int_{\R^d}\rho_T(x,\a)dx+\rho_T^\4(\c,\d)=1.
\]

Meanwhile, if $\bR^{(k)}(\cdot,\d)$, for $k\in\set{1,2,3,4}$, denotes the terminal distribution of the reference measure in the dead regime in each Scenario, then we have the following relations:
\begin{figure}[h!]
\centering
\begin{tikzpicture}[
    node distance=0.8cm and 0.3cm,
    box/.style={rectangle, rounded corners, draw, minimum height=0.8cm, minimum width=4cm, align=center},
    arrow/.style={-Stealth, thick}
]
% Nodes
\node[box] (C1) {Scenario (1): $\bR_T^\1(\db{x,\t},\d)$};
\node[box, below left=of C1] (C2) {Scenario (2): $\bR_T^\2(\db{\c,\t},\d)$};
\node[box, below right=of C1] (C3) {Scenario (3): $\bR_T^\3(x,\d)$};
\node[box, below right=of C2] (C4) {Scenario (4): $\bR_T^\4(\c,\d)$};
% Arrows
\draw[arrow] (C1) -- (C2) node[midway, above left] {Integrate in $x$ over $\R^d$};
\draw[arrow] (C1) -- (C3) node[midway, above right] {Integrate in $\t$ over $[0,T]$};
\draw[arrow] (C2) -- (C4) node[midway, below left] {Integrate in $\t$ over $[0,T]$};
\draw[arrow] (C3) -- (C4) node[midway, below right]{Integrate in $x$ over $\R^d$};
\end{tikzpicture}
\caption{Relations among the four Scenarios.}
\end{figure}
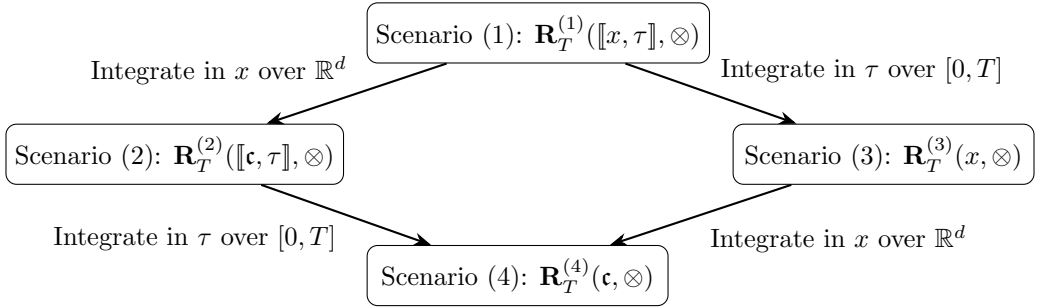

Assuming the active-regime target distributions $\rho_0$ and $\rho_T(\cdot,\a)$ are the same across all the Scenarios, and the dead-regime target distributions $\rho_T^{(k)}(\cdot,\d)$ for $k\in\set{1,2,3,4}$ satisfy the relations in the diagram above, a natural question is to compare the unbalanced Schr\"odinger bridges $\widehat\bP^{(k)}$ for $k\in\set{1,2,3,4}$ obtained in these Scenarios. We will provide a detailed comparison analysis in \cref{Subsection: Comparison}.
\subsection{Assumptions and Results}
In this subsection, we present the assumptions under which our work is conducted and we outline our main results.

\subsubsection*{Our Assumptions}
We impose a set of standard assumptions on the coefficients involved in the killing mechanism \eqref{eqn: killing mechanism} and the SDE \eqref{SDE for Xt}, as well as proper conditions on the target $(\rho_0,\rho_T)$, to ensure that the uSBP admits a solution that takes a desired form. This will allow us to apply the general theory of the SBP for regime-switching (jump) diffusions developed in \cite{zlotchevski2025schrodinger}.

First, we require the coefficients $b,\sigma,V$ that appear in \eqref{eqn: killing mechanism} and \eqref{SDE for Xt}, as well as the terminal position function $\psi$, to satisfy the following assumptions throughout this work:
\paragraph{\textbf{Assumption (C)}}\customlabel{Assumption C}{\textbf{(C)}}
\begin{enumerate}
    \item 
    $b=(b_m)_{1\leq m\leq d}:[0,T]\times\R^d \to \R^d$ is such that $b_m,\nabla b_m\in (C_b\cap\cH^\alpha_\text{u})$\footnote[2]{The definitions of ``$\cH^\alpha_\text{u}$'' and other function classes are given in \cref{subsection:notation}.}$([0,T]\times\R^d)$ (for some $\alpha>0$) for every $1\leq m\leq d$.
    \item 
    $\sigma=(\sigma_{mn})_{1\leq m,n\leq  d} : [0,T]\times\R^d \to\R^{d\times d}$ is such that, if $\sigma\sigma^\top=:a=(a_{mn})_{1\leq m,n\leq d}$, then $a_{mn}\in (C_b \cap C_\text{u})([0,T]\times\R^d)$ and $\nabla a_{mn},\nabla^2a_{mn}\in (C_b\cap\cH^\alpha_\text{u})([0,T]\times\R^d)$ for every $1\leq m,n\leq d$. Moreover, $a$ is uniformly elliptic: there exists $c>0$ such that
    \[ \sum_{m,n}a_{mn}(t,x)y_m y_n \geq c |y|^2 \text{ for all }(t,x)\in[0,T]\times\R^d \text{ and }y\in\R^d.
    \]
    \item $V:[0,T]\times\R^d\to\R_+$ is in $ (C_b\cap\cH^\alpha_\text{u})([0,T]\times\R^d)$ and not identically 0.
    \item $\psi:[0,T]\times\R^d\to\R^d$ is measurable and locally bounded.
\end{enumerate}
One purpose of Assumption \ref{Assumption C} is to guarantee the well-definedness of our model. Namely, under Assumption \ref{Assumption C}, we can always find a strong Markov reference measure $\bR\in\cP(\Omega)$, with any pre-set initial distribution, such that the canonical process $\set{(X_t,\Lambda_t)}_{[0,T]}$ under $\bR$ satisfy both the killing mechanism \eqref{eqn: killing mechanism} and the SDE \eqref{SDE for Xt}. Second, backed by the classical regularity theory for parabolic partial differential equations (PDEs), Assumption \ref{Assumption C} is also sufficient for the existence of a positive transition density function under $\bR$ within the active regime. Moreover, this transition density function is known to satisfy both the forward and the backward PDEs corresponding to the uSBP. All the above provides the foundation for our study of the dynamics of the Schr\"odinger bridge $\widehat\bP$. We will be more specific on the implications of Assumption \ref{Assumption C} in \cref{subsection:impl. of assump.}. \\

Next, to ensure that the uSBP (\ref{dfn: SBP-dynamic})  admits a unique solution, we impose:
\vspace{0.25cm}
\paragraph{\textbf{Assumption (A)}}\customlabel{Assumption A}{\textbf{(A)}} The target distributions $(\rho_0,\rho_T)$ satisfy $\KL{\rho_0\otimes\rho_T}{\bR_{0T}}<\infty$.
\vspace{0.2cm}\\ 
This assumption can be relaxed to the existence of some $\bP\in\cP(\Omega)$ such that $\bP_0=\rho_0$, $\bP_T=\rho_T$, and $\KL{\bP_{0T}}{\bR_{0T}}<\infty$. However, in practice, the statement in Assumption \ref{Assumption A} is easier to verify. 

We note that by properties of the KL divergence, Assumption \ref{Assumption A} implies that $$\KL{\rho_0}{\bR_0}<\infty\;\text{ and }\KL{\rho_T}{\bR_T}<\infty.$$
Thus, in particular, it guarantees that $\rho_0\ll\bR_0$ and $\rho_T\ll\bR_T$.

It is known (see, e.g. \cite{leonard2014survey}) that, if Assumptions \textup{\ref{Assumption A}} and \textup{\ref{Assumption C}} are satisfied, then the SBP \eqref{dfn: SBP-dynamic} admits a unique solution $\widehat\bP$ such that $\KL{\widehat\bP}{\bR}<\infty$, and further this optimizer takes the form 
\begin{equation}\label{eqn:bP^ as prod of fg}
    \widehat\bP=\f(X_0,\Lambda_0)\g(X_T,\Lambda_T)\bR, 
\end{equation} 
where $\f,\g :\R^d\times\set{\a,\d} \to \R_+$ are non-negative measurable functions that solve\footnote[3]{In fact, $\f,\g$ are the unique solutions to this system up to a scaling constant.} the \textit{static Schr\"odinger system}:
    \begin{equation}\label{Equation: fg Schrodinger System}
        \begin{dcases}
        \vspace{0.2cm} \f(x,\a)\Esub{\bR}{\g(X_T,\Lambda_T)|(X_0,\Lambda_0)=(x,\a)} = \frac{d\rho_0}{d\bR_0}(x,\a)& \text{ for $\bR_0$-a.e. $(x,\a)$},\\
        \g(x,\a)\Esub{\bR}{\f(X_0,\Lambda_0)|(X_T,\Lambda_T)=(x,\a)} = \frac{d\rho_T}{d\bR_T}(x,\a)& \text{ for $\bR_T$-a.e. $(x,\a)$},\\
        \g(y,\d)\Esub{\bR}{\f(X_0,\Lambda_0)|(X_T,\Lambda_T)=(y,\d)} = \frac{d\rho_T}{d\bR_T}(y,\d)& \text{ for $\bR_T$-a.e. $(y,\d)$}.
        \end{dcases}
    \end{equation} 
We note that, in the setup of our uSBP, $\bR_T(\cdot,\d)$ will be supported on the range of the function $\psi$, and hence the last equation of \eqref{Equation: fg Schrodinger System} can vary greatly with different choices of $\psi$; in \cref{Subsection: static system}, we will discuss this aspect in the four Scenarios.\\ 

Finally, in order to establish the ``dynamic'' analog of the system \eqref{Equation: fg Schrodinger System} and to take full advantage of analytic and stochastic tools in exploring the properties of $\widehat\bP$, we impose the following assumptions on $\f,\g$. 
\paragraph{\textbf{Assumption (B)}}\customlabel{Assumption B}{\textbf{(B)}}
\begin{enumerate}%[label=\textbf{(B)}]
    \item The function $\f$ satisfies the integrability condition $\int_{\R^d}\f(x,\a)\bR_0(dx,\a)<\infty$. 
    \item The function $\g$ is bounded on the support of $\bR_T$.
\end{enumerate}
Under the collective effect of Assumptions \ref{Assumption A}, \ref{Assumption B}, \ref{Assumption C}, following the general theory developed in \cite{zlotchevski2025schrodinger}, we will be able to study the Schr\"odinger bridge $\widehat\bP$ using stochastic calculus and analytic tools. Let us now give an overview of our results.

\subsubsection*{Main Results}
Our work is composed of two main parts. In the first part, we derive various fundamental properties of the unbalanced Schr\"odinger bridge $\widehat\bP$ for diffusion with killing by adapting the general theory from \cite{zlotchevski2025schrodinger} to our uSBP setting. To begin, in \cref{section: general results}, we discuss the aspects of $\widehat\bP$ that apply to all choices of $\psi$, even beyond the four Scenarios outlined in \cref{Subsection: Scenarios description}. Notably, we show that $\widehat\bP$ is the path measure of another diffusion with killing. This part of the result is centered around the pair of functions $\varphi$ and $\widehat\varphi$ defined in \eqref{Definition: varphi-usbp} and \eqref{dfn: phi-hat active regime} which, as we discuss in \cref{section: case-by-case results}, are solutions to the \textit{dynamic Schr\"odinger system} corresponding to the uSBP. Based on $\varphi$ and $\widehat\varphi$, we present (in \textbf{\cref{thm: P-hat SDE and generator}}) the explicit formulation of the SDE associated with $\widehat\bP$ in the active regime, along with the new killing rate, and we further obtain (in \textbf{\cref{thm: transition and active marginals}}) the transition density function under $\widehat\bP$. We also derive the \textit{stochastic control} formulation of the uSBP and its solution (in \textbf{\cref{def:stochastic control}} and \textbf{\cref{thm:stochastic control}}). These are ``universal'' properties of $\widehat\bP$: although they depend on $\psi$ (through $\varphi$), they do not vary from one choice of $\psi$ to another. 

In \cref{section: case-by-case results}, we study facets of the uSBP that must be analyzed on a case-by-case basis depending on the specific choice of $\psi$; although the expressions yielded in this section still depend explicitly on $\psi$, they vary from Scenario to Scenario. The results obtained in Scenarios (3) and (4) generalize those of the previous uSBP works \cite{zlotchevski2025schrodinger} and \cite{chen2022most} respectively. In our Scenario (3), we consider a time and space dependent killing rate $V(t,x)$ in contrast with the space-homogeneous killing rate $V(t)$ in \cite{zlotchevski2025schrodinger}. In our Scenario (4), we obtain additional insight on the uSBP compared to \cite{chen2022most}, such as the operator approach to the dynamic Schr\"odinger system and Kolmogorov forward equation. \\

In the second part, we analyze connections among our different Scenarios as well as relations to other uSBP works. In \cref{Subsection: Comparison}, we compare uSBP Scenarios in which $(\bR,\rho_0,\rho_T)$ are the ``same'', in the sense that $b,\sigma, \bR_0, V, \rho_0, \rho_T(\cdot,\a)$ are identical, and $\bR_T(\cdot,\d),\rho_T(\cdot,\d)$ satisfy the diagram at the end of \cref{subsection:four scenarios} across the Scenarios. We demonstrate that the Scenario with a ``stronger'' constraint yields a Schr\"odinger bridge with greater KL divergence to the reference measure; for example, Scenario (1) has a stronger constraint compared to Scenario (2), as $\bR_T^{(1)} (\cdot,\d)$ and $\rho_T^{(1)} (\cdot,\d)$ respectively have $\bR_T^{(2)}(\cdot,\d)$ and $\rho_T^{(2)}(\cdot,\d)$ as the marginal, and in this case we have  $\KL{\widehat\bP^\1}{\bR^\1}\geq \KL{\widehat\bP^\2}{\bR^\2}$. We also determine (in \textbf{\cref{Thm: KL divergence comparison}}) the conditions under which the above equality is achieved. In particular, we show that the equality holds for a unique choice of  $\rho_T(\cdot,\d)$. These conditions take a simple form in the case where $\bR_0$ is a Dirac delta distribution. 

In \cref{Section: connection to other paper}, we examine the connections of our uSBP model to the uSBP studied in \cite{eldesoukey2024excursion}. The latter is a variant of our Scenario (1) where the constraint $\widehat\bP_T(\cdot,\d)=\rho_T(\cdot,\d)$ is removed. We demonstrate (in \textbf{\cref{thm: comparison to other paper}}) that the unbalanced Schr\"odinger bridge in this case inherits properties from our Scenario (1) and that the Schr\"odinger system solution satisfies $\g(\cdot,\a)\equiv 1$. This recovers the uSBP results obtained in \cite{eldesoukey2024excursion}, further illustrating the versatility of the regime-switching approach adopted in this work.

\subsection{Notation and Terminology}\label{subsection:notation}
We denote by $C_b([0,T]\times\R^d)$ the class of continuous and bounded functions on $[0,T]\times\R^d$, by $C_\text{u}([0,T]\times\R^d)$ the class of functions that are uniformly continuous (in both variables) on $[0,T]\times\R^d$. For $\alpha>0$, we write $\cH^\alpha_\text{u}([0,T]\times\R^d)$ for the class of functions that are H\"older-$\alpha$ continuous in the spatial variable uniformly on $[0,T]\times\R^d$; i.e., $g\in\cH^\alpha_\text{u}([0,T]\times\R^d)$ if there exists a constant $C=C_{\alpha,g}>0$ such that \[
|g(t,x)-g(t,y)|\leq C|x-y|^\alpha\textup{ for all }x,y\in\R^d \text{ and }t\in[0,T].
\] 

For a function $f:(t,x)\in[0,T]\times\R^d\mapsto f(t,x)\in\R$, the notations $\nabla f,\nabla^2 f$ refer to the gradient and the Hessian respectively of $f$ with respect to the spatial variable $x$ only. We say that such a function $f$ is \textit{of class $C^{1,2}$} if $f\in C([0,T]\times\R^d)$, and as a function on $(0,T)\times\R^d$, $f$ is once differentiable in the temporal variable and twice differentiable in the spatial variable with $\frac{\partial}{\partial t}f,\nabla f,\nabla^2 f\in C((0,T)\times\R^d)$.
If $f$ is of class $C^{1,2}$ and in addition, $f$ is supported on a compact subset of $(0,T)\times\R^d$, then  $f$ is said to be \textit{of class $C^{1,2}_c$}. We call $f$ \textit{of class $C_b^{1,2}$}, if $f$ is of class $C^{1,2}$, and as a function on $(0,T)\times\R^d$, $\frac{\partial}{\partial t}f, \nabla f,\nabla^2 f$ are bounded in the spatial variable, with the bounds being uniform in the temporal variable over any compact subset of $(0,T)$. 

For a function $f:[0,T]\times\R^d\times\rSet\to\R$, we say that $f$ is of class $C^{1,2}$ (resp. $C^{1,2}_c$, $C^{1,2}_b$), if, as a function on $[0,T]\times\R^d$, $f(\cdot,\a)$ is of class $C^{1,2}$ (resp. $C^{1,2}_c$, $C^{1,2}_b$) (and we do not need these properties for $f(\cdot,\d)$).\\

The dimension $d$ of the spatial component is used loosely: in Scenarios (1) and (2), the true dimension is in fact $d+1$. The state space of paths, domains of functions, etc., should be understood as $\R^d$ or $\R^{d+1}$ depending on the Scenario, with $X_t$ replaced by $\db{X_t,\t_t}$ if needed. However, since only the active regime is concerned, the transition density $q(t,x,s,y)$ as in \eqref{eqn: Definition of q} is always defined for $x,y\in\R^d$, with the trivial augmented coordinate omitted. 

For a function $f$ that appears across multiple uSBP Scenarios,
we will write $f(t,x,\a)$ in the active regime, which is to be understood as $x\in\R^d$ in Scenarios (3), (4), and as $x\in\R^{d+1}$ in Scenarios (1), (2), with the last coordinate being trivial. If needed, we may also write $\aug{x}$ for the trivial augmentation of $x\in\R^d$ to a $(d+1)$-dimensional vector.

In the dead regime, we use the notation $f(t,y,\d)$ for $y\in\textup{Range}(\psi)$ or $y\in\R^d$ or $y\in\R^{d+1}$; if we need to address the augmented coordinate separately, we will write $f(t,\db{x,\t},\d)$, where $x\in\R^d$ and $\t \in [0,T]$.

We use the superscript ``$^{(k)}$'', for $k\in\set{1,2,3,4}$, to indicate that a certain set, measure or function is associated with Scenario $(k)$.
The expression $f(t,\psi(\t,x),\d)$, where $t,\t \in[0,T]$ and $x\in\R^d$, should be understood as, respectively in Scenario (1), (2), (3), (4):
\begin{enumerate}
    \item $f^\1(t,\db{x,\t},\d)$, for $(t,x,\t)\in[0,T]\times \R^d \times[0,T]$,
    \item $f^\2(t,\db{\c,\t},\d)$, for $(t,\t)\in[0,T]\times[0,T]$,
    \item $f^\3(t,x,\d)$, for $(t,x)\in[0,T]\times \R^d$,
    \item $f^\4(t,\c,\d)$, for $t\in[0,T]$.
\end{enumerate}

\section{Unbalanced Schr\"odinger Bridge - General Results}\label{section: general results}
In this section, we will provide an overview of the properties of the unbalanced Schr\"odinger bridge $\widehat\bP$ solving \eqref{dfn: SBP-dynamic}. Throughout this section, we assume that Assumptions \ref{Assumption A}, \ref{Assumption B}, and \ref{Assumption  C} are in force. 

We begin with a brief review of the implications of our assumptions which will be relevant to later discussions. Then we study properties of the uSBP that are common across all choices of the function $\psi$ satisfying Assumption \ref{Assumption C}. The expressions in this section depend explicitly on $\psi$ and they hold true beyond the four Scenarios described in \cref{Subsection: Scenarios description}. 

\subsection{Implication of the Assumptions}\label{subsection:impl. of assump.}
In this subsection we will outline some useful properties and facts implied by our assumptions. 

First, under Assumption \ref{Assumption C}, given any probability distribution $\mu_0$ on $\R^d$, there exists a strong Markov path measure $\bR\in\cP(\Omega)$ with $\bR_0=\mu_0\times\delta_{\a}$ such that the canonical process $\set{(X_t,\Lambda_t)}_{[0,T]}$ under  $\bR$ satisfies  the SDE \eqref{SDE for Xt} and the killing mechanism \eqref{eqn: killing mechanism}. Further, it is a well known result in the regularity theory of parabolic PDEs (see, e.g., \cite{friedman2008partial,friedmanstochastic}) that such a path measure $\bR$, when restricted within the active regime, admits a \textit{transition density function} in the spatial component in the sense that there exists an everywhere positive function $q(t,x,s,y)$ for $x,y\in\R^d$ and $0\leq t < s \leq T$ such that, for every Borel $B\subseteq \R^d$,
\begin{equation}\label{eqn: Definition of q}
    \bR(X_s \in B, \Lambda_s =\a | (X_t, \Lambda_t) =(x,\a)) = \int_B q(t,x,s,y)dy.
\end{equation}
Based on $q(t,x,s,y)$, we can also derive the cross-regime transition density function as:
\begin{align*}
 \bR(X_{\kt-}\in B,\Lambda_s=\d |(X_t,&\Lambda_t) =(x,\a)) \\&= \bR(X_{\kt-}\in B, \kt\in(t,s]| (X_t,\Lambda_t) =(x,\a))\\
 &=\int_{B}\int^s_t V(r,y)q(t,x,r,y)drdy   
\end{align*}
where, again, $\kt$ denotes the killing time; this implies that the regime switch $\a\to\d$ has the transition density function 
\begin{equation}\label{eqn:cross-regime density}
  \tilde q (t,x,s,y):=V(s,y)q(t,x,s,y) 
\end{equation}
which is the probability that an active particle starting from location $x$ at time $t$ gets killed at location $y$ at a later time $s$.\\

Next, in view of \eqref{eqn: killing mechanism} and \eqref{SDE for Xt}, 
we define, for $f:[0,T]\times\R^d\times\rSet \to \R$ 
of class $C^{1,2}$ and $(t,x)\in[0,T]\times\R^d$, the operators 
\begin{equation*}%\label{eq:def of L0 usbp}
    \begin{aligned}
        L_0f(t,x)&:=b(t,x)\cdot\nabla f(t,x,\a)+\frac{1}{2}\sum_{m,n}a_{mn}(t,x)\frac{\partial^2f}{\partial x_m \partial x_n}(t,x,\a)
    \end{aligned}
\end{equation*}
and 
\begin{equation}\label{dfn: L}
Lf(t,x,i):=\begin{dcases}
   L_0f(t,x)+V(t,x)(f(t,\psi(t,x),\d)-f(t,x,\a)) & \text{if }i=\a,\\
   0 & \text{if }i=\d.
\end{dcases}
\end{equation}
We also refer to $L$ as the \textit{generator} of the diffusion with killing. The $L^2([0,T]\times\R^d)$-adjoint of $L_0$, denoted by $L_0^*$, is defined for $f$ of class $C^{1,2}$ as:
\begin{equation}\label{eq:def of L*0 usbp}
    \begin{aligned}
        L^*_0f(t,x)&:=-\nabla \cdot(b f)(t,x,\a)+\frac{1}{2}\sum_{m,n}\frac{\partial^2(a_{mn}f)}{\partial x_m \partial x_n}(t,x,\a).
    \end{aligned}
\end{equation}
Then, the transition density function $q$ found above satisfies the following \textit{backward} and \textit{forward} PDEs in terms of $L_0$ and $L_0^*$ respectively:
\begin{enumerate}
    \item for every $(s,y)\in(0,T]\times\R^d$, the function $(t,x)\in(0,s)\times\R^d\mapsto q(t,x,s,y)$ is of class $C^{1,2}_b$ 
    and satisfies 
\begin{equation*}%\label{eq: backward equation with killing}
    \frac{\partial}{\partial t}q(t,x,s,y)=(-L_0+V)q(\cdot,s,y) \,(t,x);
\end{equation*}
    \item for every $(t,x)\in[0,T)\times\R^d$, the function $(s,y)\in(t,T)\times\R^d\mapsto q(t,x,s,y)$ is of class $C^{1,2}_b$
    and satisfies 
\begin{equation*}%\label{eq: forward equation with killing}
    \frac{\partial}{\partial s}q(t,x,s,y)=(L_0^*-V)q(t,x,\cdot) \,(s,y).    
\end{equation*}
\end{enumerate}

Finally, we are ready to define the pair of functions $(\varphi,\widehat\varphi)$ that plays a central role in the theory of the SBP. The first function $\varphi:[0,T]\times\R^d\times\set{\a,\d}\to\R_+$ is given by
\begin{equation}\label{dfn: general varphi}
    \varphi(t,x,i):=\bE_{\bR}\brac{\g(X_T,\Lambda_T)|(X_t,\Lambda_t)=(x,i)}\textup{ for }(t,x,i)\in[0,T]\times\R^d\times\set{\a,\d}.
\end{equation}
In particular, $\varphi(T,\cdot)=\g$, and for $t\in[0,T)$, by the definition of $q$,
\begin{equation}\label{Definition: varphi-usbp}
\begin{dcases}
\varphi(t,x,\a):=\int_{\R^d}q(t,x,T,z)\g(z,\a)dz\\
\hspace{2cm}+\int_{\R^d}\int_t^T q(t,x,r,z)V(r,z)\g(\psi(r,z),\d)dr dz\;\;\textup{for }x\in\R^d,\\
\varphi(t,y,\d):=\g(y,\d)=\varphi(T,y,\d)\;\;\textup{for }y\in\textup{Range}(\psi)\textup{ (which is constant in $t$)}.
\end{dcases}
\end{equation} 

As a consequence of Assumptions  \ref{Assumption B} and \ref{Assumption  C}, the following holds. 
\begin{prop}\label{proposition : harmonic varphi}
    Let $L$ be the operator given by \eqref{dfn: L}. As a function on $[0,T]\times\R^d$, $\varphi(\cdot,\a)$ is of class $C^{1,2}$. In addition,
    \begin{enumerate}
        \item $\varphi(\cdot,\a)>0$ on $(0,T)\times\R^d$, and 
        \item $\varphi$ is \textit{harmonic} under $L$ in the sense that $(\frac{\partial}{\partial t}+L) \varphi=0$ on $(0,T)\times\R^d \times \rSet$.
    \end{enumerate}
\end{prop}
As for the second function $\widehat\varphi$, for now we only give its definition in the active regime: for $x\in\R^d$ and $t\in(0,T]$,
\begin{equation}\label{dfn: phi-hat active regime}
\widehat\varphi(t,x,\a):=\int_{\R^d}q(0,z,t,x)\f(z,\a)\bR_0(dz,\a)=\int_{\R^d}q(0,z,t,x)\f(z,\a)\mu_0(dz)
\end{equation}
where we recall that $\bR_0=\mu_0\times\delta_{\a}$, and $\widehat\varphi(0,dx,\a):=\f(x,\a)\mu_0(dx)$ as a measure on $\R^d$. Note that, combining \eqref{dfn: general varphi} with \eqref{Equation: fg Schrodinger System}, we also have 
\begin{equation}\label{eqn:varphi*varphi^ at 0}
\varphi(0,x,\a)\widehat\varphi(0,dx,\a)=\rho_0(dx).   
\end{equation}
The definitions of $\widehat\varphi(\cdot,\d)$ will be given in \cref{section: case-by-case results} in a case-by-case manner for the four Scenarios. 

\subsection{General Results on the uSBP} 
Using the general theory of the regime-switching SBP \cite{zlotchevski2025schrodinger}, we can immediately obtain some results on the unbalanced Schr\"odinger bridge $\widehat\bP$. Recall that the killing rate $V$ is the regime-switch rate from regime $\a$ to regime $\d$, and that the switching rate $\d \to \a$ is identically 0. Again, Assumptions \ref{Assumption A}, \ref{Assumption B}, and \ref{Assumption C} are assumed to hold in this subsection.
\subsubsection{The unbalanced Schr\"odinger bridge SDE, killing rate, and generator}
We first present the diffusion with killing formulation of the unbalanced Schr\"odinger bridge. The following theorem is a direct translation of the general result proven in \cite{zlotchevski2025schrodinger}.
\begin{thm}\cite[Theorem 2.2.1]{zlotchevski2025schrodinger}\label{thm: P-hat SDE and generator}
Let $\bR,\rho_0,\rho_T,\varphi$ be as above, i.e., $\bR\in\cP(\Omega)$ with $\bR_0=\mu_0\times\delta_{\a}$ for some probability distribution $\mu_0$ on $\R^d$, whose canonical process solves \eqref{eqn: killing mechanism} and \eqref{SDE for Xt}, $\rho_0,\rho_T$ probability distributions on $\R^d\times\set{\a,\d}$ such that Assumptions \textup{\ref{Assumption A}} and \textup{\ref{Assumption B}} are satisfied, $\varphi$ defined as in \eqref{Definition: varphi-usbp}. Denote by $\kt$ the killing time, and recall that $\widehat\bP$, the solution to \eqref{dfn: SBP-dynamic}, takes the form of \eqref{eqn:bP^ as prod of fg}. Then, the canonical process $\proc$ under $\widehat\bP$ is again a diffusion process with killing, where the spatial component $X_t$ satisfies in the active regime the SDE: 
\begin{equation*}%\label{eqn: P-hat active SDE}
 \begin{aligned}   dX_t&= (b+a\nabla\log\varphi)(t,X_{t})dt +\sigma(t,X_{t})dB_t; \quad 0 \leq t < \kt \wedge T, \quad X_0\sim\rho_0
\end{aligned}
\end{equation*}
and the killing rate is $\displaystyle \frac{\varphi(t,\psi(t,x),\d)}{\varphi(t,x,\a)}V(t,x)$. In other words,
the generator of $\proc$ under $\widehat\bP$, denoted by $L_{\widehat\bP}$, is given by, for $f$ of class $C^{1,2}_c$ and $(t,x)\in[0,T]\times\R^d$:
\begin{equation}\label{eqn: bP-hat generator}
\begin{aligned}
        L_{\widehat\bP} f(t,x,\a)&=(b+a\nabla\log \varphi)(t,x,\a)\cdot \nabla f(t,x,\a) \\ 
        &\hspace{1cm}+ \frac{1}{2}\sum_{m,n} a_{mn}\frac{\partial^2 f}{\partial x_m \partial x_n}(t,x,\a)\\
        &\hspace{1cm}+ \frac{\varphi(t,\psi(t,x),\d)}{\varphi(t,x,\a)}V(t,x)\pran{f(t,\psi(t,x),\d)-f(t,x,\a)},
\end{aligned}
\end{equation}
and $L_{\widehat\bP} f(t,y,\d)=0$ for all $t\in[0,T]$, $y\in \textup{Range}(\psi)$.
\end{thm}
\subsubsection{The unbalanced Schr\"odinger bridge transition and marginal probabilities} With the product form \eqref{eqn:bP^ as prod of fg} of $\widehat\bP$ and the introduction of $\varphi,\widehat\varphi$ in \eqref{Definition: varphi-usbp} and \eqref{dfn: phi-hat active regime}, we can also determine the transition density function, as well as the marginal density functions, of $\widehat\bP$. The formulas below have also been established in \cite{zlotchevski2025schrodinger}.
\begin{thm}\cite[Theorems 3.1.2 and 3.2.4]{zlotchevski2025schrodinger}\label{thm: transition and active marginals}
    Under the same setting as in \cref{thm: P-hat SDE and generator}, the unbalanced Schr\"odinger bridge $\widehat\bP$ admits a transition density function $\widehat q(t,x,s,y)$ for $x,y\in\R^d$ and $0\leq t < s \leq T$ (in the sense of \eqref{eqn: Definition of q}) given by
    \begin{equation}\label{eqn: p-hat transition density}
            \widehat q(t,x,s,y):=\frac{\varphi(s,y,\a)}{\varphi(t,x,\a)}q(t,x,s,y).
    \end{equation}
    Further, for every $t\in(0,T]$, the marginal distribution $\widehat\bP_t$, i.e., the distribution of $(X_t,\Lambda_t)$ under $\widehat\bP$, possesses in the active regime the density function
\begin{equation}\label{eq:marginal density active regime}
    \widehat\bP_t(x,\a):=\varphi(t,x,\a)\widehat\varphi(t,x,\a)\;\textit{ for } x\in\R^d
\end{equation}
with $\widehat\varphi$ defined as in \eqref{dfn: phi-hat active regime}, and for $t\in(0,T)$, it satisfies the Kolmogorov forward equation
\begin{equation}\label{eq:P-hat forward equation active regime}
    \begin{aligned}
            \frac{\partial }{\partial t}\widehat\bP_t(x,\a)&=\nabla \cdot ((b+a\nabla \log\varphi)\widehat\bP_t)(t,x,\a) \\
            &\hspace{0.5cm}-\frac{1}{2} \sum_{m,n}\frac{\partial^2(a_{mn}\widehat\bP_t)}{\partial x_m \partial x_n}(t,x,\a)-\frac{\varphi(t,\psi(t,x),\d)}{\varphi(t,x,\a)}V(t,x)\widehat\bP_t(x,\a).
            \end{aligned}
\end{equation}
\end{thm}
As we will see in the next section, the statement \eqref{eq:marginal density active regime} will hold in the dead regime $\d$ once we properly define $\widehat\varphi(\cdot, \d)$ for each choice of $\psi$.
\subsubsection{Stochastic control formulation for the uSBP}
To develop the stochastic control formulation for our uSBP, we further assume that, for every compact set $K\subseteq\R^d$, 
    \begin{equation}\label{eqn:cond on phi-stoch contrl}
     \int_0^T\sup_{x\in K}\abs{\nabla\log\varphi(t,x,\a)}^2dt<\infty\textup{ and }\int_0^T\sup_{x\in K}\abs{\frac{\varphi(t,\psi(t,x),\d)}{\varphi(t,x,\a)}}^2dt<\infty.
    \end{equation}
Note that \cref{proposition : harmonic varphi} already implies that $\nabla\log\varphi(t,x,\a)$ and $\frac{1}{\varphi(t,x,\a)}$ are bounded uniformly in $(t,x)$ on any compact subset of $[0,T)\times\R^d$. Thus, heuristically speaking, the conditions above extend the ``boundedness'' to the endpoint $t=T$ in the integral form.

Let $\cU$ be the class of pairs 
$(u,\xi)$ such that:
\begin{enumerate}
    \item $u:[0,T]\times\R^d \to \R^d$ and $\xi:[0,T]\times\R^d \to (0,\infty)$ are measurable functions,
    \item For every compact set $K\subseteq \R^d$, $\int_0^T\sup_{x\in K}(|u(t,x)|^2+|\xi(t,x)|^2)dt<\infty$.
    \item There exists a path measure $\bP^{(u,\xi)}\in\cP(\Omega)$ with $\bP^{(u,\xi)}_0=\rho_0$, and the  canonical process $\proc$ under $\bP^{(u,\xi)}$ is a diffusion with killing satisfying
    \[
    dX_t=[b(t,X_t)+\sigma(t,X_t)u(t,X_t)]dt+\sigma(t,X_t)dB_t,\;0\leq t<\kt\wedge T,
    \]
where $\kt$ is the killing time, and such that the killing rate is $\xi(t,x)V(t,x)$.
\end{enumerate}
We refer to the above SDE and killing rate under $\bP^{(u,\xi)}$ as the \textit{controlled SDE} and the \textit{controlled killing rate}. It is clear that the generator corresponding to $\bP^{(u,\xi)}$, denoted by $L_{\bP^{(u,\xi)}}$, is given by: for $f$ of class $C_c^{1,2}$ and $(t,x)\in[0,T]\times\R^d$,
\begin{align*}
        L_{\bP^{(u,\xi)}} f(t,x,\a)&:=[b(t,x)+\sigma(t,x)u(t,x)]\cdot \nabla f(t,x,\a) \\
        & \hspace{0.5cm} +\frac{1}{2}\sum_{m,n} a_{mn}(t,x)\frac{\partial^2 f}{\partial x_m \partial x_n}(t,x,\a)\\
        & \hspace{0.5cm} + \xi(t,x)V(t,x)\pran{f(t,\psi(t,x),\d)-f(t,x,\a)},\\
        L_{\bP^{(u,\xi)}} f(t,x,\d)&\equiv 0.
\end{align*}
Finally, we also request $(u,\xi)\in\cU$ to satisfy that
\begin{enumerate}[label=(4)]
    \item under $\bP^{(u,\xi)}$, 
    $\displaystyle
    \set{-\log\varphi(t,X_t,\Lambda_t)-\int_0^t\bigg(\frac{\partial}{\partial r}+L_{\bP^{(u,\xi)}}\bigg)(-\log\varphi)(r,X_r,\Lambda_r)dr}_{[0,T]}
    $
    is a martingale with $\varphi$ defined as in \eqref{Definition: varphi-usbp}, and
\end{enumerate}
\begin{enumerate}[label=(5)]
    \item under $\bP^{(u,\xi)}$, the distribution of $(X_T,\Lambda_T)$  is equal to $\rho_T$. \\
\end{enumerate}

\noindent \textbf{Remark. }We want to point that a particular case in which Condition (3) is fulfilled is when $\bP^{(u,\xi)}$ is a Girsanov transform of the (modified) reference measure $\bR$ associated with the original SDE and the original killing rate. Namely, let $\proc$ be the canonical process under $\bR$ and define
    \begin{equation*}
        \begin{aligned}
          Z_T^{(u,\xi)}:=\exp\bigg(\int_0^T&u(t,X_t)dB_t-\frac{1}{2}\int_0^T|u(t,X_t)|^2dt\\
          &+\int_0^T\int_{\R_+}\log \xi(t,X_t)\one_{[0,V(t,X_t))]}(w)\tilde N_1(dt,dw)\\
          &+\int_0^TV(t,X_t)[\log\xi(t,X_t)+1-\xi(t,X_t)]dt\bigg),
        \end{aligned}
    \end{equation*}
where $\tilde N_1(dt,dw)$ is the compensated Poisson random measure with intensity being the Lebesgue measure on (a bounded subset of) $[0,T]\times\R_+$. Note that Conditions (1), (2) guarantee that all the (stochastic) integrals involved in $Z^{(u,\xi)}_T$ are well-defined. Then, Condition (3) is satisfied if $\bE_\bR[Z^{(u,\xi)}_T\frac{d\rho_0}{d\mu_0}(X_0)]=1$ and $\bP^{(u,\xi)}:=Z^{(u,\xi)}_T\,\frac{d\rho_0}{d\mu_0}(X_0)\bR$; in other words, $\bP^{(u,\xi)}$ is the Girsanov transform (by $u$ and $\xi$) of $\bP^0:=\frac{d\rho_0}{d\mu_0}(X_0)\bR$ which is the modification of $\bR$ with the initial distribution modified from $\mu_0$ to $\rho_0$.\\

Applying the regime-switching SBP theory from \cite[Section 2.3]{zlotchevski2025schrodinger}, the uSBP is equivalent to the following stochastic control problem.
\begin{dfn}\label{def:stochastic control}
   The stochastic control formulation of the uSBP is to determine
\begin{equation}\label{stochastic control objective}
\begin{aligned}
    (u^*,\xi^*)=\arg\min_{(u,\xi)\in \cU} &\bE_{\bP^{(u,\xi)}}\bigg[\int_0^T\bigg(\frac{1}{2}\abs{u(t,X_t)}^2\\
    &+V(t,X_t)\big(\xi(t,X_t)\log\xi(t,X_t)+1-\xi(t,X_t)\big)\bigg)dt\bigg].
\end{aligned}
\end{equation}
\end{dfn}
\noindent The expectation on the right hand side of \eqref{stochastic control objective} coincides with $\KL{\bP^{(u,\xi)}}{\bP^0}$ when $u,\xi$ are sufficiently ``nice'' (e.g., bounded) \cite[Section 2.2]{zlotchevski2025schrodinger}. We note that the quadratic cost in terms of $u$ for modifying the drift coefficient is identical to the classic SBP for diffusions. In the uSBP setting, an additional cost is incurred in terms of $\xi$ for modifying the killing rate. 
\begin{thm}\cite[Theorem 2.3.3]{zlotchevski2025schrodinger}\label{thm:stochastic control}
    The optimal controls for the stochastic control problem \eqref{stochastic control objective} are given by, for $(t,x)\in(0,T)\times \R^d$,
\begin{equation}\label{eqn: optimal controls}
\begin{dcases}
u^*(t,x)=\sigma^\top\nabla\log \varphi(t,x,\a),\\
\xi^*(t,x) = \frac{\varphi(t,\psi(t,x),\d)}{\varphi(t,x,\a)}.
\end{dcases}
\end{equation}
\end{thm}
Note that \cref{proposition : harmonic varphi} and the condition \eqref{eqn:cond on phi-stoch contrl} on $\varphi$ guarantee that $(u^*,\xi^*)$ given in \eqref{eqn: optimal controls} are indeed in $\cU$. As expected, this choice recovers the form of the unbalanced Schr\"odinger bridge described above, with $\bP^{(u^*,\xi^*)}=\widehat\bP$.
\section{The Four Scenarios: Case-by-Case Results}\label{section: case-by-case results}
In the previous section, we have derived general results on the uSBP that hold for any choice of $\psi$ and thus cover a wide range of uSBP scenarios. However, there remain certain aspects of the uSBP which need to be treated on a case-by-case basis. For example, the analogue of \eqref{eq:marginal density active regime} in the dead regime cannot be stated in general terms because there is no ``universal equivalent'' to \eqref{dfn: phi-hat active regime} for $\widehat\varphi(t,x,\d)$. 
In this section, we will cover Scenarios (1), (2), (3), (4) as described in \cref{Subsection: Scenarios description}; other choices of $\psi$ may be approached by the same techniques with some adaptations. We will continue working under Assumptions \ref{Assumption A}, \ref{Assumption B}, and \ref{Assumption C} throughout this section.
\subsection{The Marginal Distribution of the Killed Particles}
We begin by discussing the function $\widehat\varphi$ in regime $\d$. Based on the theory of the regime-switching SBP \cite{zlotchevski2025schrodinger}, the function $\widehat\varphi$ is generally defined as
\begin{equation}\label{dfn: phi-hat general expectation}
    \widehat\varphi(t,x,i)=\bE_{\bR}\brac{\f(X_0,\Lambda_0);(X_t,\Lambda_t)=(x,i)}.
\end{equation}
Thus, we expect $\widehat\varphi(t,\cdot)$ to be a ``forward evolution'' of $\widehat\varphi(0,\cdot):=\f\bR_0$, and that expressions such as \eqref{eq:marginal density active regime} hold in every regime. On the other hand, the marginal distribution $\widehat\bP_t(\cdot,\d)$ is directly influenced by the choice of $\psi$. In particular, its support is the set $\textup{Range}(\psi)$, and thus $\widehat\bP_t(\cdot,\d)$ is not guaranteed to have a Lebesgue density. Recall that we are interested in four uSBP scenarios:
\begin{enumerate}%[label=\arabic*)]
    \item $\psi^\1(t,x)=\db{x,t}$, 
    \item  $\psi^\2(t,x)=\db{\c,t}$, where $\c\in\R^d$ is fixed, 
    \item $\psi^\3(t,x)=x$, 
    \item $\psi^\4(t,x)=\c$. 
\end{enumerate}
Inspecting \eqref{dfn: phi-hat general expectation}, for $i=\d$ and $x\in \textup{Range}(\psi)$, 
this expression for $\widehat\varphi(t,x,\d)$ can be viewed as ``adding up'' (or integrating) the values of \[\int_{\R^d}q(0,z,r,y)V(r,y) \f(z,\a)\mu_0(dz)\] over the set \[\set{(r,y)\in (0,t]\times\R^d : \psi(r,y)=x}.\]
Hence, for each of the choices of $\psi$ above, define respectively, for $t,\t \in(0,T]$ and $x\in\R^d$:
\begin{enumerate}
    \item $\widehat\varphi^\1(t,\db{x,\t},\d)= \one_{t \geq \t}\int_{\R^d}q(0,z,\t,x)V(\t,x) \f(z,\a)\mu_0(dz)$,
    \item $\widehat\varphi^\2(t,\db{\c,\t},\d)= \one_{t \geq \t}\int_{\R^d}\int_{\R^d}q(0,z,\t,y)V(\t,y) \f(z,\a)\mu_0(dz)dy$,
    \item $\widehat\varphi^\3(t,x,\d)=\int_0^t\int_{\R^d}q(0,z,r,x)V(r,x) \f(z,\a)\mu_0(dz)dr$,
    \item $\widehat\varphi^\4(t,\c,\d)=\int_{\R^d}\int_0^t\int_{\R^d}q(0,z,r,y)V(r,y) \f(z,\a)\mu_0(dz)drdy$;
\end{enumerate}
by \eqref{dfn: phi-hat active regime}, the expressions above can also be written as 
\begin{equation}\label{eqs: phi-hat killed regime}
    \begin{dcases}
        \widehat\varphi^\1(t,\db{x,\t},\d)= \one_{t \geq \t}V(\t,x)\widehat\varphi(\t,x,\a),\\
        \widehat\varphi^\2(t,\db{\c,\t},\d)= \one_{t \geq \t}\int_{\R^d}V(\t,y)\widehat\varphi(\t,y,\a)dy,\\
        \widehat\varphi^\3(t,x,\d)=\int_0^t V(r,x)\widehat\varphi(r,x,\a)dr, \\
        \widehat\varphi^\4(t,\c,\d)=\int_{\R^d}\int_0^t V(r,y)\widehat\varphi(r,y,\a)drdy.
    \end{dcases}
\end{equation}
For $t=0$ or $\t=0$, we identify $\widehat\varphi(\cdot,\d)\equiv0$ in all Scenarios. 
\begin{thm}
    For $t\in(0,T]$, the marginal distribution of the unbalanced Schr\"odinger bridge in the dead regime, denoted by $\widehat\bP^{(k)}_t(\cdot,\d)$ in Scenario $(k)$, for $k\in\set{1,2,3,4}$, admits a density function
    \begin{equation}\label{eq: marginal density killed regime}
        \widehat\bP^{(k)}_t(y,\d)=\varphi^{(k)}(t,y,\d)\widehat\varphi^{(k)}(t,y,\d) \; \textup{ for }y \in \textup{Range}(\psi^{(k)}),
    \end{equation}
    in the sense specified in Scenarios (1), (2), (3), (4) respectively as:
    \begin{enumerate}
    \item $\widehat\bP^\1_t(\db{x,\t},\d)$ is a density with respect to the Lebesgue measure on $\R^d\times[0,T]$,
    \item $\widehat\bP^\2_t(\db{\c,\t},\d)$ is a density with respect to the Lebesgue measure on $[0,T]$,
    \item $\widehat\bP^\3_t(x,\d)$ is a density with respect to the Lebesgue measure on $\R^d$,
    \item $\widehat\bP^\4_t(\c,\d)$ is a Dirac delta distribution at $\set{\c}$.
\end{enumerate}
Further, in Scenarios (1) and (2), we have that for all $t\in[0,T],x\in\R^d,\t\in[0,T]$,
\begin{equation}\label{eqn: Scenarios 1,2 killed forward}
\begin{aligned}
    \widehat\bP^\1_t(\db{x,\t},\d)&=\one_{t\geq\t}\;\rho_T(\db{x,\t},\d),\\
    \widehat\bP^\2_t(\db{\c,\t},\d)&=\one_{t\geq\t}\;\rho_T(\db{\c,\t},\d).
\end{aligned}
\end{equation}
In Scenario (3), the density $\widehat\bP^\3_t(\cdot,\d)$ satisfies for $t\in(0,T)$ the forward equation
\begin{equation}\label{eqn: Scenario 3 killed forward}
    \frac{\partial }{\partial t}\widehat\bP^\3_t(x,\d)=\frac{\varphi^\3(t,x,\d)}{\varphi^\3(t,x,\a)}V(t,x)\widehat\bP^\3_t(x,\a) \text{ for Lebesgue-a.e. }x\in\R^d,
\end{equation}
and in Scenario (4), the scalar $\widehat\bP^\4_t(\c,\d)$ satisfies for $t\in(0,T)$ the forward equation
\begin{equation}\label{eqn: Scenario 4 killed forward}
    \frac{\partial }{\partial t}\widehat\bP^\4_t(\c,\d)=\int_{\R^d}\frac{\varphi^\4(t,\c,\d)}{\varphi^\4(t,x,\a)}V(t,x)\widehat\bP^\4_t(x,\a)dx.
\end{equation}
\end{thm}
\begin{proof}
    Once \eqref{eq: marginal density killed regime} is established, the relations \eqref{eqn: Scenarios 1,2 killed forward} can be derived by observing that for $k=1,2$,
    \[
    \varphi^{(k)}(t,y,\d)\widehat\varphi^{(k)}(t,y,\d)=\one_{t\geq\t}\;\varphi^{(k)}(T,y,\d)\widehat\varphi^{(k)}(T,y,\d)=\one_{t\geq\t}\;\rho_T^{(k)}(y,\d).
    \]
    The last part of the theorem (equations \eqref{eqn: Scenario 3 killed forward} and \eqref{eqn: Scenario 4 killed forward}) can be obtained by computing the time derivative $\frac{\partial}{\partial t}(\varphi\widehat\varphi)$, recalling $\frac{\partial}{\partial t}\varphi(\cdot,\d)=0$, using the expressions of $\widehat\varphi(\cdot,\d)$ above and simplifying using \cref{eq:marginal density active regime}.

Let us prove that Equation \eqref{eq: marginal density killed regime} holds in all four Scenarios. By \cref{thm: P-hat SDE and generator} and \cref{thm: transition and active marginals}, we have the transition density and the killing rate of the unbalanced Schr\"odinger bridge $\widehat\bP$, and we can therefore compute the marginal probability distribution. Let $\kt$ be the killing time. Then, given $0<t\leq T$, for every Borel $B\subseteq \R^d$ and $0< s'<s\leq t$,
\begin{equation}\label{eqn:joint dist of killing location+time under bP^}
    \begin{aligned}
        \widehat\bP&(X_{\kt-} \in B,\kt\in(s',s])\\&=\int_{s'}^s\int_{B}\int_{\R^d} \frac{\varphi(r,y,\a)}{\varphi(0,x,\a)}q(0,x,r,y) \frac{\varphi(r,\psi(r,y),\d)}{\varphi(r,y,\a)}V(r,y)\rho_0(dx)dydr\\
        &=\int_{s'}^s\int_{B}\varphi(r,\psi(r,y),\d) V(r,y)\int_{\R^d} q(0,x,r,y)\widehat\varphi(0,dx,\a)dydr\\
        &=\int_{s'}^s\int_{B}\varphi(r,\psi(r,y),\d) V(r,y)\widehat\varphi(r,y,\a)dydr
    \end{aligned}
\end{equation}
where we also used \eqref{dfn: phi-hat active regime} and \eqref{eqn:varphi*varphi^ at 0}. The joint distribution of the killing locations and the killing times under $\widehat\bP$ is fully determined by \eqref{eqn:joint dist of killing location+time under bP^}. As a result, recalling that $\varphi(\cdot,\d)$ is constant in time and applying the specific $\psi$ in Scenarios (1), (2), (3), (4) respectively, we obtain, for all $x\in\R^d$ and $\t \in [0,T]$:
\begin{enumerate}
    \item $\widehat\bP^\1_t(\db{x,\t},\d)=\one_{t \geq \t}\;\varphi^\1(t,\db{x,\t},\d) V(\t,x)\widehat\varphi(\t,x,\a)$,
    \item $\widehat\bP^\2_t(\db{\c,\t},\d)=\one_{t \geq \t}\;\varphi^\2(t,\db{\c,\t},\d) \int_{\R^d}V(\t,y)\widehat\varphi(\t,y,\a)dy$,
    \item $\widehat\bP^\3_t(x,\d)=\varphi^\3(t,x,\d) \int_{0}^{t}V(r,x)\widehat\varphi(r,x,\a)dr$,
    \item $\widehat\bP^\4_t(\c,\d)=\varphi^\4(t,\c,\d) \int_{\R^d}\int_{0}^{t}V(r,y)\widehat\varphi(r,y,\a)dr dy$.
\end{enumerate}
Hence, observing each corresponding equation \eqref{eqs: phi-hat killed regime}, we arrive at \eqref{eq: marginal density killed regime}.
    
\end{proof}
\subsection{The Static Schr\"odinger System and the Fortet-Sinkhorn Algorithm}\label{Subsection: static system}
For Scenarios (1), (2), (3), (4), the last equation in the static Schr\"odinger system \eqref{Equation: fg Schrodinger System} is respectively:
\begin{enumerate}
    \item  For $\bR_T$-a.e. $(\db{x,\t},\d)$, where $\db{x,\t}\in\R^d\times[0,T]$,
    \begin{equation}\label{eqn: Scenario (1) third equation Schrodinger system}
        \g(\db{x,\t},\d)\Esub{\bR}{\f(X_0,\Lambda_0)|(\db{X_T,\t_t},\Lambda_T)=(\db{x,\t},\d)} = \frac{d\rho_T}{d\bR_T}(\db{x,\t},\d).
    \end{equation}
    \item  For $\bR_T$-a.e. $(\db{\c,\t},\d)$, where $\c \in\R^d$ is fixed and $\t\in[0,T]$,
    \begin{equation}\label{eqn: Scenario (2) third equation Schrodinger system} 
         \g(\db{\c,\t},\d)\Esub{\bR}{\f(X_0,\Lambda_0)|(\db{X_T,\t_t},\Lambda_T)=(\db{\c,\t},\d)} = \frac{d\rho_T}{d\bR_T}(\db{\c,\t},\d).
    \end{equation}
    \item For $\bR_T$-a.e. $(x,\d)$, where $x\in\R^d$,
    \[
    \g(x,\d)\Esub{\bR}{\f(X_0,\Lambda_0)|(X_T,\Lambda_T)=(x,\d)} = \frac{d\rho_T}{d\bR_T}(x,\d).
    \]
    \item At the point $(\c,\d)$,
    \[
    \g(\c,\d)\Esub{\bR}{\f(X_0,\Lambda_0)|(X_T,\Lambda_T)=(\c,\d)} = \frac{d\rho_T}{d\bR_T}(\c,\d).
    \]
\end{enumerate}

Consequently, following a standard argument (see e.g. \cite{zlotchevski2025schrodinger}), the static Schr\"odinger system \eqref{Equation: fg Schrodinger System} can be rewritten as, for Lebesgue-a.e. $x\in\R^d$ and $y\in\textup{Range}(\psi)$, 
\begin{equation}\label{eqn: phi-phihat static system}
\left\{
    \begin{aligned}
        &\begin{aligned}
    \varphi(0,{x},\a)&= \int_{\R^d} q(0,x,T,z)\varphi(T,{z},\a)dz\\
    &\hspace{3cm}+\int_{\R^d}\int_0^T q(0,x,\t,z)V(\t,z)\varphi(T,\psi(\t,z),\d)d\t dz,
    \end{aligned}\\      
    &\varphi(0,y,\d)=\varphi(T,y,\d),\\
    &\widehat\varphi(T,{x},\a)=\int_{\R^d} q(0,z,T,x)\widehat\varphi(0,dz,\a),\\ 
    &\widehat\varphi(T,y,\d)= \text{see equations below.}\\
    &\varphi(0,x,\a)\widehat\varphi(0,dx,\a)=\rho_0(dx)\;\text{as a measure on }\R^d,\\
    &\varphi(T,\cdot)\widehat\varphi(T,\cdot)=\rho_T(\cdot).
    \end{aligned}
\right.
\end{equation}
By \eqref{eqs: phi-hat killed regime}, the expression for $\widehat\varphi(T,y,\d)$ with $y\in\textup{Range}(\psi)$, in Scenarios (1), (2), (3), (4), is respectively:
\begin{enumerate}
    \item $\widehat\varphi^\1(T,\db{x,\t},\d)=\int_{\R^d}q(0,z,\t,x)V(\t,x) \widehat\varphi(0,dz,\a)$,
    \item $\widehat\varphi^\2(T,\db{\c,\t},\d)=\int_{\R^d}\int_{\R^d}q(0,z,\t,y)V(\t,y) \widehat\varphi(0,dz,\a)dy$,
    \item $\widehat\varphi^\3(T,x,\d)=\int_0^T\int_{\R^d}q(0,z,r,x)V(r,x) \widehat\varphi(0,dz,\a)dr$,
    \item $\widehat\varphi^\4(T,\c,\d)=\int_{\R^d}\int_0^T\int_{\R^d}q(0,z,r,y)V(r,y) \widehat\varphi(0,dz,\a)drdy$.
\end{enumerate}

When $\mu_0$ admits a density function $\mu_0(x)$ with respect to the Lebesgue measure (and hence so does $\rho_0$, i.e., $\rho_0(dx)=\rho_0(x)dx$), $\widehat\varphi(0,x,\a)=\f(x,\a)\mu_0(x)$ is then defined as a function on $\R^d$. In this case, the above form \eqref{eqn: phi-phihat static system} of the static Schr\"odinger system allows one to compute $(\varphi,\widehat\varphi)$ via a computationally efficient algorithm called the \textit{Fortet-Sinkhorn algorithm}. It consists of a loop that alternates through the integrals and boundary constraints of \eqref{eqn: phi-phihat static system}:
\begin{equation*}
    {
    \widehat\varphi(T,\cdot)
    }
        \mapsto
    {
    \varphi(T,\cdot)
    }
        \mapsto
    {
    \varphi(0,\cdot)
    }
        \mapsto
    {
    \widehat\varphi(0,\cdot)
    }
        \mapsto
    {
    \widehat\varphi(T,\cdot)_{\text{next}}.
    }
\end{equation*}
It is shown in \cite{chen2022most} that for Scenario (4), if $(x,t,y)\in\R^d\times(0,T] \times \R^d\mapsto q(0,x,t,y)$ is continuous and positive on compact sets, and $(t,x)\in[0,T]\times\R^d\mapsto V(t,x)$ is continuous and not identically zero, then the Fortet-Sinkhorn algorithm converges to the unique solution of the static Schr\"odinger system. 
The argument can be easily adapted to show that under the same conditions on $q$ and $V$, the convergence result also holds in Scenarios (1), (2), and (3). We do not expand on this point here and refer interested readers to \cite{chen2016entropic,chen2022most,eldesoukey2024excursion,eldesoukey2025inferring}. 

\subsection{The Dynamic Schr\"odinger System and the Kolmogorov Forward Equation}
Let us first recall some universal (for all $\psi$) properties of $\varphi(\cdot,\a)$ and $\widehat\varphi(\cdot,\a)$. By \cref{proposition : harmonic varphi}, $\varphi$ is harmonic, meaning that, for every $t\in(0,T)$, $x\in\R^d$ and $y\in\textup{Range}(\psi)$,
\begin{equation*}
\begin{dcases}
        \frac{\partial\varphi }{\partial t}(t,x,\a)=-L_0\varphi(t,x,\a)-V(t,x)(\varphi(t,\psi(t,x),\d)-\varphi(t,x,\a)),\\
       \frac{\partial \varphi}{\partial t}(t,y,\d)=0.
\end{dcases}
\end{equation*} 
In the same vein, using the property of $q$ combined with the integrability Assumption \ref{Assumption B}, it is straightforward to deduce that, for every $t\in(0,T)$ and $x\in\R^d$, 
\[
\pran{-\frac{\partial }{\partial t}+ L^*_0}\widehat\varphi(t,x,\a)-V(t,x)\widehat\varphi(t,x,\a)=0,
\]
where $L^*_0$ is defined by \eqref{eq:def of L*0 usbp}.\\

Combining these facts with the expressions for $\widehat\varphi(\cdot,\d)$ obtained above, we have the following relations in Scenario $(k)$, $k=1,2,3,4$: for $t\in(0,T),x\in\R^d, y\in\textup{Range}(\psi)$,
\begin{equation*}\label{Dynamic Schrodinger System}
    \begin{dcases}
    \begin{aligned}
     \pran{\frac{\partial }{\partial t}+L_0}&\varphi(t,x,\a)\\&+V(t,x)(\varphi(t,\psi(t,x),\d)-\varphi(t,x,\a))=0,   
    \end{aligned}
    \;&\frac{\partial \varphi}{\partial t}(t,y,\d)=0.\\
    \pran{-\frac{\partial }{\partial t}+ L^*_0}\widehat\varphi(t,x,\a)-V(t,x)\widehat\varphi(t,x,\a)=0,\;& \widehat\varphi(t,y,\d)= \text{Equations } \eqref{eqs: phi-hat killed regime}.\\
    \varphi(0,x,\a)\widehat\varphi(0,dx,\a)=\rho_0(dx),\;&\varphi(0,y,\d)\widehat\varphi(0,y,\d)=0.\\
    \varphi(T,x,\a)\widehat\varphi(T,x,\a)=\rho_T(x,\a),\;&\varphi(T,y,\d)\widehat\varphi(T,y,\d)=\rho_T(y,\d).
    \end{dcases}
\end{equation*}
This system can be seen as the time-continuum extension of \eqref{eqn: phi-phihat static system}, and we therefore refer to it as the \textit{dynamic Schr\"odinger system}.
In Scenarios (3) and (4), we obtain additional information on the dynamic aspect of $\varphi$ and $\widehat\varphi$.

First, let us recall that Scenario (3) has already been treated in \cite{zlotchevski2025schrodinger}, in which the above system is compiled into the Kolmogorov backward and forward equations, equipped with a pair of boundary conditions. We state this result below without proof.
\begin{thm}\label{thm: dynamic system with L-star}\cite[Section 3.2]{zlotchevski2025schrodinger}
    In Scenario (3), the pair $(\varphi,\widehat\varphi)$ satisfies the dynamic Schr\"odinger system
    \begin{equation}\label{eqn: dynamic system using L-star}
    \begin{dcases}
    \begin{aligned}
       &\pran{\frac{\partial}{\partial t}+L} \varphi(t,x,i) = 0\\
       &\pran{-\frac{\partial}{\partial t}+L^*} \widehat\varphi(t,x,i) = 0
    \end{aligned}
        &\textit{ for } (t,x,i)\in(0,T)\times\R^d\times\rSet,\\
     \begin{aligned}
     &\varphi(0,x,i)\widehat\varphi(0,dx,i)=\rho_0(dx,i)\\
     &\varphi(T,x,i)\widehat\varphi(T,x,i)=\rho_T(x,i)
     \end{aligned}  & \textit{ for Lebesgue-a.e. }x\in\R^d\textit{ and }i\in\rSet,
    \end{dcases}
\end{equation}
where $L$ is the operator \eqref{dfn: L} with $\psi(t,x)=x$ and $L^*$ is its adjoint operator with respect to the inner product
\begin{equation}\label{dfn: Scenario 3 inner product}
    \inn{f,g} = \int_0^T\int_{\R^d} f(t,x,\a)g(t,x,\a)dxdt+\int_0^T\int_{\R^d} f(t,x,\d)g(t,x,\d)dxdt.
\end{equation}
Moreover, if $L_{\widehat\bP}$ is the operator given by \eqref{eqn: bP-hat generator} with $\psi(t,x)=x$, then in Scenario (3), the marginal distribution $\widehat\bP_t$ satisfies for $t\in(0,T)$ and $(x,i)\in \R^d \times \set{\a,\d}$,
    \begin{equation}\label{eqn: forward equation marginal density L-star}
        \frac{\partial }{\partial t}\widehat\bP_t(x,i) = L_{\widehat\bP}^*\widehat\bP_t(x,i),
    \end{equation}
    where $L_{\widehat\bP}^*$ is the adjoint of $L_{\widehat\bP}$ with respect to the inner product \eqref{dfn: Scenario 3 inner product}.
\end{thm}

As for \sce{4} (the case $\psi(t,x)=\c$), it recovers the dynamic Schr\"odinger system derived in \cite{chen2022most}. However, by virtue of the regime-switching approach, we demonstrate below that with a properly chosen inner product, the dynamic Schr\"odinger system in \sce{4} can \textit{also} be expressed in the adjoint-operator form \eqref{eqn: dynamic system using L-star}.
\begin{lem}
Suppose that $\psi(t,x)\equiv\c$ and define
\begin{equation}\label{dfn: Scenario 4 inner product}
    \inn{f,g}= \int_0^T\int_{\R^d} f(t,x,\a)g(t,x,\a)dxdt + \int_0^T f(t,\c,\d)g(t,\c,\d)dt
\end{equation}
for functions $f,g:[0,T]\times\R^d\times\rSet\to\R_+$ such that the above quantity is finite. 
Let $L$ be as in \eqref{dfn: L} and let $f,g$ be any functions of class $C^{1,2}_c$. Then $\inn{Lf,g}=\inn{f,L^*g}$, where $L^*g$ is given by
\begin{equation}\label{L-star}
\begin{aligned}
    L^*g(t,x,\a)&=\nabla \cdot (bg)(t,x,\a) -\frac{1}{2} \sum_{m,n}\frac{\partial^2(a_{mn}g)}{\partial x_m \partial x_n}(t,x,\a) - V(t,x)g(t,x,\a),\\
    L^*g(t,\c,\d)&=\int_{\R^d} V(t,y)g(t,y,\a)dy.
\end{aligned}
\end{equation}
\end{lem}
\begin{proof}
For simplicity we suppress the $t$ variable in $f,g$ as well as the $dt$-integral, thus getting
\begin{align*}
    \inn{Lf, g} &= \inn{L_0f(\cdot,\a),g(\cdot,\a)}_{L^2} + \int_{\R^d} V(t,x)(f(\c,\d)-f(x,\a))g(x,\a)dx + \underbrace{Lf(\c,\d)}_{=0}g(\c,\d)\\
    &= \inn{f(\cdot,\a),L_0^*g(\cdot,\a)}_{L^2} - \int_{\R^d}V(t,x)f(x,\a)g(x,\a)dx + f(\c,\d)\int_{\R^d} V(t,x)g(x,\a)dx\\
    &= \inn{f(\cdot,\a),L_0^*g(\cdot,\a)-V(t,\cdot)g(\cdot,\a)}_{L^2}  + f(\c,\d)\int_{\R^d} V(t,x)g(x,\a)dx.
\end{align*} 
By the definition \eqref{dfn: Scenario 4 inner product}, the result follows.
\end{proof}
As an immediate consequence, observing equation \eqref{eqs: phi-hat killed regime}, we obtain the following.
\begin{thm}
    In Scenario (4), where $\psi(t,x)=\c$, the dynamic Schr\"odinger system is given by 
\begin{equation*}%\label{eqn: dynamic system Scenario (4)}
    \begin{dcases}
        \pran{\frac{\partial}{\partial t}+L} \varphi(t,x,\a) = 0 &  \pran{\frac{\partial}{\partial t}+L} \varphi(t,\c,\d) = 0,\\
        \pran{-\frac{\partial}{\partial t}+L^*} \widehat\varphi(t,x,\a) = 0 & \pran{-\frac{\partial}{\partial t}+L^*} \widehat\varphi(t,\c,\d) = 0, \\
        \rho_0(dx,\a) = 
        \varphi(0,x,\a)\widehat\varphi(0,dx,\a) & 
        \rho_0(\c,\d) = \varphi(0,\c,\d)\widehat\varphi(0,\c,\d)=0,\\
        \rho_T(x,\a) = \varphi(T,x,\a)\widehat\varphi(T,x,\a) & 
        \rho_T(\c,\d) = \varphi(T,\c,\d)\widehat\varphi(T,\c,\d).
    \end{dcases}
\end{equation*}  
    where $L^*$, defined in \eqref{L-star}, is the adjoint operator of $L$ with respect to the action $\inn{\cdot,\cdot}$ defined in \eqref{dfn: Scenario 4 inner product}, and the first two lines of equations hold for every $(t,x)\in(0,T)\times\R^d$, the last two lines for Lebesgue-a.e. $x\in\R^d$.
\end{thm}
\noindent We can see that this dynamic Schr\"odinger system is the analogue of \eqref{eqn: dynamic system using L-star} with the state space in the dead regime being $\set{\c}$ instead of $\R^d$.\\

With this inner product in hand, we can also derive the operator formulation of the Kolmogorov forward equations \eqref{eq:P-hat forward equation active regime} and \eqref{eqn: Scenario 4 killed forward}.

\begin{thm}\label{thm: dynamic system Scenario (4)}
    Let $L_{\widehat\bP}$ be the operator given by \eqref{eqn: bP-hat generator} with $\psi(t,x)=\c$, and let $L_{\widehat\bP}^*$ be the adjoint of $L_{\widehat\bP}$ with respect to the inner product \eqref{dfn: Scenario 4 inner product}. Then in Scenario (4), the active-regime marginal distribution density $\widehat\bP_t(x,\a)$ satisfies
    \begin{equation*}
        \frac{\partial }{\partial t}\widehat\bP_t(x,\a) = L_{\widehat\bP}^*\widehat\bP_t(x,\a)\;\textit{ for }(t,x)\in(0,T)\times\R^d,
    \end{equation*}
    and the dead-regime mass $\widehat\bP_t(\c,\d)$ satisfies 
    \begin{equation*}
        \frac{\partial }{\partial t}\widehat\bP_t(\c,\d) = L_{\widehat\bP}^*\widehat\bP_t(\c,\d)\;\textit{ for }t\in(0,T).
    \end{equation*}
\end{thm}
It is expected that no analogue of \eqref{eqn: forward equation marginal density L-star} exists in Scenarios (1) and (2) because for all $t\in[0,T]$, the marginal distribution $\widehat\bP_t(\cdot,\d)$ is entirely prescribed by $\rho_T(\cdot,\d)$ and only varies in $t$ through the indicator function $\one_{t\geq \t}$, as per \eqref{eqn: Scenarios 1,2 killed forward}. \\

\noindent \textbf{Remark.} By the general theory of regime-switching (jump) diffusions, it is known that \cref{thm: dynamic system with L-star} will hold for any choice of $\psi$ such that the map $x\mapsto \psi(t,x)$ is a $C^1$-diffeomorphism of $\R^d$ \cite{zlotchevski2025schrodinger}.
On the other hand, as seen for the case $\psi(t,x)=\c$ in \cref{thm: dynamic system Scenario (4)}, statements of the form \eqref{eqn: dynamic system using L-star} and \eqref{eqn: forward equation marginal density L-star} are not limited to diffeomorphism $\psi$. It is possible to define an appropriate inner product for the specific choice of $\psi$ and to express the dynamic Schr\"odinger system and Kolmogorov forward equation in the operator form. 

\section{Comparison Among Different uSBP Scenarios and Their Solutions}\label{Subsection: Comparison}
In this section we discuss the relations among the four Scenarios presented in the previous sections. For the sake of clarity, in this section we use the superscript ``$^{(k)}$'', for $k\in\set{1,2,3,4}$, in all the relevant quantities to indicate the Scenario to which they correspond. Again, all the discussions are conducted under Assumptions \ref{Assumption A}, \ref{Assumption B}, and \ref{Assumption C}. We will also assume that the reference measure $\bR$ is the same in all four Scenarios, in the sense that the initial distribution $\bR_0(\cdot,\a)$, the drift and diffusion coefficients $b(t,x),\sigma(t,x)$ in the regime $\a$, as well as the killing rate $V(t,x)$, are identical from Scenario to Scenario. 

Recall that these four Scenarios represent different levels of information regarding the killed particles up to time $T$. This naturally leads to a comparison question in an environment where two or more observers have access to different levels of information. For example, suppose that Observer 1 has access to the full joint distribution of the killing locations and the killing times $\rho^\1_T(\db{x,\t},\d)$ as in Scenario (1), and Observer 2 can only see the temporal distribution $\rho^\2_T(\db{\c,\t},\d)$ as in Scenario (2), with $\rho^\2_T(\db{\c,\t},\d)= \int_{\R^d}\rho^\1_T(\db{x,\t},\d)dx$ for every $\t\in(0,T]$. What is the difference between their respective solutions $\widehat\bP^\1$ and $\widehat\bP^\2$, obtained using the uSBP models described in the previous section? Will Observer 2 be able to use the uSBP setup in Scenario (1)?

Let us first present an informal, intuitive argument to motivate our results. The uSBP in \sce{2} imposes less of a constraint on the terminal distribution than that in \sce{1}, and therefore we would expect the relation $\KL{\widehat\bP^\2}{\bR^\2}\leq \KL{\widehat\bP^\1}{\bR^\1}$. In fact, finding $\widehat\bP^\2$ in \sce{2} can be viewed as a ``SBP within a SBP'': every joint distribution of the killing locations and the killing times imposed through the constraint $\rho_T^\1(\db{x,\t},\d)$ yields a path measure $\widehat\bP^\1$, and we must choose the one minimizing KL divergence while having $ \int_{\R^d}\rho^\1_T(\db{x,\t},\d)dx=\rho^\2_T(\db{\c,\t},\d)$ for every $\t\in(0,T]$. Meanwhile, even though the terminal distribution $\widehat\bP_T^\2(\cdot,\d)$ only concerns the killing times, there exists an ``unobserved'' joint distribution of killing locations and killing times\footnote[4]{It can be explicitly computed using the transition density \eqref{eqn: p-hat transition density} and the initial distribution $\bR_0$.} corresponding to $\widehat\bP^\2$. If we were then to choose $\rho^\1_T(\cdot,\d)$ to be precisely this joint distribution, it seems reasonable that the Schr\"odinger bridge $\widehat\bP^\1$ would be in some sense ``identical'' to $\widehat\bP^\2$. We make this discussion precise in the following subsection. 
\subsection{KL Divergence Comparison Between Scenario (1) and Scenario (2)}
To begin, let us recall the concept of \textit{pushforward measure} and its properties, which will be used extensively in our comparison of different uSBP scenarios.
\begin{dfn}
Let $Y,Z$ be Polish spaces. Let $\mu$ be a probability measure on $(Y,\cB_Y)$ and let $f:Y\to Z$ be a measurable function. The pushforward of $\mu$ by $f$, denoted by $f_{\#}\mu$, is the measure defined by, for $B\in\cB_Z$,
\[
    f_{\#}\mu(B)= \mu(f^{-1}(B)).
\]
\end{dfn}
If $\mu,\nu$ are two probability measures on $(Y,\cB_Y)$, then we also have that
\begin{equation}\label{eqn: pushforward density}
\frac{df_{\#}\mu}{df_{\#}\nu} (z)= \bE_{\nu}\brac{\frac{d\mu}{d\nu} \bigg| f=z},
\end{equation}
and
\begin{equation}\label{eqn: KL of pushforward inequality}
    \KL{\mu}{\nu} \geq \KL{f_{\#}\mu}{f_{\#}\nu}.
\end{equation}
The first statement is a general fact, and the latter is a consequence of the KL divergence additive property (\cite{leonard2014survey} - formula A.8). The Skorokhod space (equipped with the Skorokhod metric) is indeed a Polish space, and thus the properties above apply to the path measures we consider.\\

We provide a detailed analysis of the comparison between Scenario (1) and Scenario (2). Other comparisons will follow by similar arguments. We write $\Omega^\1 :=\Omega(\psi^\1)$ and $\Omega^\2 :=\Omega(\psi^\2)$ for the path spaces in Scenarios (1) and (2) respectively, as defined in \cref{dfn:omega}, where a sample path is of the form $\omega(t)=(\db{X_t,\t_t},\Lambda_t)$, for $t\in[0,T]$. To be rigorous, we will also include the trivial augmented coordinate in the active regime. Without loss of generality, we assume $\t\equiv0$ in the active regime, writing $(\db{X_t,\t_t},\Lambda_t)=(\aug{x},\a)$. 

Define the map $\Phi : \Omega^\1 \to \Omega^\2 $ as follows:
\begin{equation}\label{dfn:Phi}
\Phi(\omega)(t)=
\begin{dcases}
(\db{X_t,\t_t},\Lambda_t) & \text {if }\Lambda_t=\a,   \\  
(\db{\c,\t_t},\Lambda_t) & \text {if }\Lambda_t=\d.   
\end{dcases} 
\end{equation}
Intuitively, $\Phi$ maps a path from Scenario (1) to its corresponding path in Scenario (2). By definition of $\psi^\1$ and $\psi^\2$, $\Phi$ is a bijection. 
To describe its inverse map, let $\omega$ be a sample path in $\Omega^\2$: if $\Lambda_T(\omega)=\a$, i.e., the path survives up to time $T$, then $\Phi^{-1}(\omega)=\omega$; if  $\Lambda_T(\omega)=\d$, in which case the killing time $\kt(\omega)=\inf\set{t>0:\Lambda_t(\omega)=\d}\leq T$, then
\begin{equation*}%\label{dfn:Phi-inverse}
\Phi^{-1}(\omega)(t)=
\begin{dcases}
(\db{X_t,\t_t},\Lambda_t) & \text {if }\Lambda_t=\a\;\text{ i.e., }0\leq t<\kt(\omega),   \\  
(\db{X_{\kt-},\t_t},\Lambda_t) & \text {if }\Lambda_t=\d\;\text{ i.e., } \kt(\omega)\leq t\leq T.   
\end{dcases} 
\end{equation*}
Since $\kt$ is a stopping time, both $\Phi$ and $\Phi^{-1}$ are measurable mappings. 

We now define the reference measures and target distributions in our comparison. Suppose that $\set{(\db{X_t,\t_t},\Lambda_t)}$ under $\bR^\1$, in the active regime, solves the reference SDE \eqref{SDE for Xt} in Scenario (1) with coefficients $b^\1,\sigma^\1,V^\1$ and initial data $\bR_0^\1(\cdot,\a)$. Then clearly, $\set{(\db{X_t,\t_t},\Lambda_t) }$ under $\bR^\2:=\Phi_{\#}\bR^\1$ also solves \eqref{SDE for Xt} in Scenario (2) with the same coefficients and initial data, i.e., $b^\2=b^\1,\sigma^\2=\sigma^\1,V^\2=V^\1$, and $\bR_0^\2(\cdot,\a)=\bR_0^\1(\cdot,\a)$.
\begin{lem}\label{lem:KL of pushforward by Phi}
    For any path measure $\bP\in\cP(\Omega)$ such that $\bP\ll\bR^\1$,
    \[
    \KL{\bP}{\bR^\1}= \KL{\Phi_{\#}\bP}{\Phi_{\#}\bR^\1}= \KL{\Phi_{\#}\bP}{\bR^\2}.
    \]
\end{lem}
\begin{proof}
    Applying \eqref{eqn: KL of pushforward inequality} immediately gives
    \[
    \KL{\bP}{\bR^\1} \geq \KL{\Phi_{\#}\bP}{\Phi_{\#}\bR^\1}.
    \]
    Moreover, $\Phi$ is a bijection, and thus $\bP=\Phi^{-1}_{\#}(\Phi_{\#}\bP)$. Therefore applying \eqref{eqn: KL of pushforward inequality} again,
    \[
    \KL{\Phi_{\#}\bP}{\Phi_{\#}\bR^\1} \geq \KL{\Phi^{-1}_{\#}(\Phi_{\#}\bP)}{\Phi^{-1}_{\#}(\Phi_{\#}\bR^\1)}=\KL{\bP}{\bR^\1}.
    \]
    
\end{proof}
\begin{thm}\label{Thm: KL divergence comparison} 
    Consider the uSBPs for $\bR^\1$ and $\bR^\2$ as above with target distributions that satisfy $\rho_0^\1=\rho_0^\2$, $\rho_T^\1(\cdot,\a)=\rho_T^\2(\cdot, \a)$  and $\int_{\R^d}\rho_T^\1(\db{x,\t},\d)dx=\rho_T^\2(\db{\c,\t},\d)$ for all $\t\in[0,T]$. 
    Suppose that their respective solutions $\widehat\bP^\1$ and $\widehat\bP^\2$ exist and satisfy $$\widehat\bP^{(k)}=\f^{(k)}(\db{X_0,\t_0},\Lambda_0)\g^{(k)}(\db{X_T,\t_T},\Lambda_T)\bR^{(k)},\;\textit{ for }k=1,2,$$ 
    where $(\f^\1,\g^\1)$ and $(\f^\2,\g^\2)$ are the unique solutions of their respective static Schr\"odinger systems.
    Then
    \[
    \KL{\widehat\bP^\1}{\bR^\1} \geq \KL{\widehat\bP^\2}{\bR^\2},
    \]
    and the following are equivalent:
    \begin{enumerate}[label=\textup{(\roman*)}]
        \item $\KL{\widehat\bP^\1}{\bR^\1} = \KL{\widehat\bP^\2}{\bR^\2}$.
        \item $\Phi_{\#}\widehat\bP^\1=\widehat\bP^\2$, where $\Phi$ is defined as in \eqref{dfn:Phi}.
        \item For $\bR^\1_T(\cdot,\d)$-a.e. $(\db{x,\t})$, $\g^\1(\db{x,\t},\d)$ does not depend on $x$.
        \item For $\bR^\1_T(\cdot,\d)$-a.e. $(\db{x,\t})$, the target distribution in regime $\d$ satisfies
       \begin{equation}\label{eqn:rho1 for equality}
           \begin{aligned}
               \frac{d\rho^\1_T}{d\bR^\1_T}(\db{x,\t},\d)&=
        \g^\2(\db{\c,\t},\d)\\
        &\hspace{1cm}\cdot\Esub{\bR^\1}{\f^\2(\db{X_0,\t_0},\Lambda_0)\bigg|(\db{X_T,\t_t},\Lambda_T)=(\db{x,\t},\d)}.
           \end{aligned} 
        \end{equation}
    \end{enumerate}
    If any of \textup{(i)}, \textup{(ii)}, \textup{(iii)}, \textup{(iv)} holds, 
    then the Schr\"odinger bridge SDE and the killing rate as in \cref{thm: P-hat SDE and generator} are identical for $\widehat\bP^\1$ and $\widehat\bP^\2$.
\end{thm}
\noindent \textbf{Remark. }Before we present the proof, let us comment on the conditions in the theorem. For both $k=1,2$, under Assumption \ref{Assumption A}, the Schr\"odinger bridge $\widehat\bP^{(k)}$ is guaranteed to exist. If further the transition density $q$ as in \eqref{eqn: Definition of q} exists and is everywhere positive, then by the general SBP theory \cite[Theorem 2.9]{leonard2014survey}, the product-shaped form of $\widehat\bP^{(k)}$ will hold, with $\f^{(k)}, \g^{(k)}$ being unique up to rescaling. The additional $C^{1,2}$-regularity endowed by Assumption \ref{Assumption C} is not needed. In fact, this comparison result is entirely about the \textit{static} perspective of the uSBP and does not even depend on the specific underlying SDE.

\begin{proof}
By \cref{lem:KL of pushforward by Phi},
\[
\KL{\widehat\bP^\1}{\bR^\1}=\KL{\Phi_{\#}\widehat\bP^\1}{\Phi_{\#}\bR^\1}=\KL{\Phi_{\#}\widehat\bP^\1}{\bR^\2},
\] 
and given the choice of target distributions $\rho^\1_0,\rho^\1_T,\rho^\2_0,\rho^\2_T$, 
\[
\KL{\Phi_{\#}\widehat\bP^\1}{\bR^\2}\geq\KL{\widehat\bP^\2}{\bR^\2}
\]
because $\Phi_{\#}\widehat\bP^\1$ is admissible in Scenario (2) but not necessarily optimal. Thus, $$\KL{\widehat\bP^\1}{\bR^\1}\geq\KL{\widehat\bP^\2}{\bR^\2},$$
with equality if and only if $\Phi_{\#}\widehat\bP^\1=\widehat\bP^\2$. \\

It remains to show the equivalence to statements (iii) and (iv).  
Let us first establish some general facts. Since $\widehat\bP^\1$ is the Schr\"odinger bridge, we know that $$\frac{d\widehat\bP^\1}{d\bR^\1}(\omega)=\f^\1(\omega(0))\g^\1(\omega(T))\;\textup{ for }\omega\in\Omega^{(1)}.$$
Moreover, by \eqref{eqn: pushforward density}, 
\[
\frac{d\Phi_{\#}\widehat\bP^\1}{d\Phi_{\#}\bR^\1} (z)= \bE_{\bR^\1}\brac{\frac{d\widehat\bP^\1}{d\bR^\1} \bigg| \Phi=z}.
\]
Let $z\in\Omega^\2$ be arbitrary, and $\tilde{\omega}:=\Phi^{-1}(z)\in\Omega^{(1)}$. Then
\begin{equation}\label{eqn:dPhi-P^1/dPhi-R^1}
    \frac{d\Phi_{\#}\widehat\bP^\1}{d\Phi_{\#}\bR^\1} (z) = \f^\1(\tilde{\omega}(0))\g^\1(\tilde{\omega}(T)).
\end{equation}
Moreover, $\tilde{\omega}(0)=z(0)$, and if $\tilde{\omega}(T)=(\db{x,\t},\d)$, then $z(T)=(\db{\c,\t},\d)$.\\

We now show that (ii)$\implies$(iii). Suppose that (ii) holds, meaning $\Phi_{\#}\widehat\bP^\1=\widehat\bP^\2$. Then for $\bR^\2$-a.e. path $z\in\Omega^\2$,
\[
\frac{d\Phi_{\#}\widehat\bP^\1}{d\Phi_{\#}\bR^\1} (z)= \frac{d\widehat\bP^\2}{d\bR^\2}(z)=\f^\2(z(0))\g^\2(z(T)),
\]
where $(\f^\2,\g^\2)$ solve the Schr\"odinger system in Scenario (2). Thus for $\bR^\2$-a.e. path $z\in\Omega^\2$,
\begin{equation*}
    \f^\2(z(0))\g^\2(z(T))=\f^\1(\tilde{\omega}(0))\g^\1(\tilde{\omega}(T)),
\end{equation*}
or equivalently, for $\bR^\1$-a.e. path $\omega\in\Omega^\1$,
\begin{equation*}%\label{eqn: f1g1 and f2g2 relationship}
    \f^\1(\omega(0))\g^\1(\omega(T))=\f^\2(\Phi(\omega)(0))\g^\2(\Phi(\omega)(T)).
\end{equation*}
Recall that $\Phi(\omega)(0)=\omega(0)$ and if $\omega(T)=(\db{x,\t},\d)$, then $\Phi(\omega)(T)=(\db{\c,\t},\d)$. Thus for $\bR^\1_T(\cdot,\d)$-a.e. $\db{x,\t}$, $\g^\1(\db{x,\t},\d)$ cannot depend on $x$. 
\\

We complete the proof by showing that (iii) implies (ii) and that (iii), (iv) are equivalent. Suppose that (iii) holds, i.e. $\g^\1(\db{x,\t},\d)$ does not depend on $x$. Again, take $z\in\Omega^\2$ to be arbitrary and set $\tilde\omega:=\Phi^{-1}(z)$. Observe that, since $\tilde\omega(0)=z(0)$ and $\g^\1(\tilde{\omega}(T))=\g^\1(z(T))$ (although $\tilde{\omega}(T)\neq z(T)$),  $$\f^\1(\tilde{\omega}(0))\g^\1(\tilde{\omega}(T)) = \f^\1(z(0))\g^\1(z(T)).$$ 
Hence, by \eqref{eqn:dPhi-P^1/dPhi-R^1}, for $\bR^\2$-a.e. path $z\in\Omega^\2$,
\[
\frac{d\Phi_{\#}\widehat\bP^\1}{d\bR^\2} (z) = \f^\1(\tilde{\omega}(0))\g^\1(\tilde{\omega}(T)) = \f^\1(z(0))\g^\1(z(T)).
\]
Consequently, if we can show that $(\f^\1,\g^\1)$ solves the static Schr\"odinger system \eqref{Equation: fg Schrodinger System} for $(\bR^\2,\rho_0^\2,\rho_T^\2)$, then by uniqueness of the solution, we must have $\Phi_{\#}\widehat\bP^\1 = \widehat\bP^\2$. To this end, observe that the first two equations of \eqref{Equation: fg Schrodinger System} trivially hold for $(\f^\1,\g^\1)$ because they are identical in all Scenarios. Thus, we only need to verify that $(\f^\1,\g^\1)$ satisfies the third equation in \eqref{Equation: fg Schrodinger System}, which becomes \eqref{eqn: Scenario (2) third equation Schrodinger system} in Scenario (2). Indeed, we derive as follows:
\begin{align*}
    &\g^\1(\db{\c,\t},\d)\bE_{\bR^\2}\brac{\f^\1(\db{X_0,\t_0},\Lambda_0)\bigg|(\db{X_T,\t_T},\Lambda_T)=(\db{\c,\t},\d)}\\
    &\quad=\g^\1(\db{\c,\t},\d)\Esub{\bR^\1}{\f^\1(\db{X_0,\t_0},\Lambda_0)\bigg|\t_T=\t,\Lambda_T=\d}\\
    &\quad=\frac{\g^\1(\db{\c,\t},\d)}{\int_{\R^d}\bR^\1_T(\db{x,\t},\d)dx} \int_{\R^d}\Esub{\bR^\1}{\f^\1(\db{X_0,\t_0},\Lambda_0)\bigg|(\db{X_T,\t_T},\Lambda_T)=(\db{x,\t},\d)}\\
    &\hspace{10.5cm}\cdot\bR^\1_T(\db{x,\t},\d)dx\\
    &\quad=\frac{\int_{\R^d}\rho^\1_T(\db{x,\t},\d)dx}{\int_{\R^d}\bR^\1_T(\db{x,\t},\d)dx} \text{\; (because $(\f^\1,\g^\1)$ solves \eqref{eqn: Scenario (1) third equation Schrodinger system} and $\g^\1$ is independent of $x$)}\\
    &\quad=\frac{\rho^\2_T(\db{\c,\t},\d)}{\bR_T^\2(\db{\c,\t},\d)}=\frac{d\rho^\2_T}{d\bR_T^\2}(\db{\c,\t},\d).
\end{align*}
We have confirmed that $(\f^\1,\g^\1)$ satisfies \eqref{eqn: Scenario (2) third equation Schrodinger system}, and thus it solves the static Schr\"odinger system \eqref{Equation: fg Schrodinger System} in Scenario (2). By the uniqueness of the Schr\"odinger system solution, we conclude that (ii) holds, i.e. $\Phi_{\#}\widehat\bP^\1 = \widehat\bP^\2$. 
This also indicates that for $\bR^\1$-a.e.,
\[
\f^\1(\db{X_0,\t_0},\Lambda_0)\g^\1(\db{X_T,\t_T},\Lambda_T)= \f^\2(\db{X_0,\t_0},\Lambda_0)\g^\2(\db{X_T,\t_T},\Lambda_T)
\]
Thus, Equation \eqref{eqn:rho1 for equality} becomes \eqref{eqn: Scenario (1) third equation Schrodinger system}, which is the third equation of the static Schr\"odinger system for $\bR^\1$, establishing (iii)$\implies$(iv). The reverse direction follows trivially: if (iv) holds, then \eqref{eqn:rho1 for equality} implies that $(\f^\2,\g^\2)$ solves the Schr\"odinger system for $(\bR^\1,\rho_0^\1,\rho_T^\1)$, and thus by the uniqueness of the solution, $\g^\1=c\, \g^\2$ for some $c\in\R_+$, meaning $\g^\1(\db{x,\t},\d)$ cannot depend on $x$.\\

Lastly, we discuss the dynamic aspect. 
Suppose that any of (i)-(iv) holds. Then, in particular, $\g^\1=c\, \g^\2$ for some $c\in\R_+$, and thus $\varphi^\1 = c\, \varphi^\2$ by \eqref{Definition: varphi-usbp} because $q$ only depends on the coefficients $b,\sigma,V$. This in turn yields 
\[
\nabla \log \varphi^\1(\cdot,\a) = \nabla \log \varphi^\2(\cdot,\a) \text{\; and \;} \frac{\varphi^\1(\cdot,\d)}{\varphi^\1(\cdot,\a)}=\frac{\varphi^\2(\cdot,\d)}{\varphi^\2(\cdot,\a)}
\]
As a result, by \cref{thm: P-hat SDE and generator}, both the SDE in the active regime and the killing rate under $\widehat\bP^\1$ and $\widehat\bP^\2$ are identical, and their respective generators $L_{\widehat\bP^\1}$ and $L_{\widehat\bP^\2}$ have the same coefficients (other than $\psi$).
\end{proof}
\noindent \textbf{Remark.} We will finish this subsection with a remark concerning the inequality $$\KL{\Phi_{\#}\widehat\bP^\1}{\bR^\2}\geq\KL{\widehat\bP^\2}{\bR^\2}$$ used in the proof above. In fact, by \eqref{eqn: KL of pushforward inequality}, it can be further decomposed into 
\begin{equation}\label{eqn: KL two-step inequality}
    \KL{\Phi_{\#}\widehat\bP^\1}{\bR^\2}\geq  \KL{\Phi_{\#}\widehat\bP^\1_{0T}}{\bR^\2_{0T}}  \geq \KL{\widehat\bP^\2_{0T}}{\bR^\2_{0T}} = \KL{\widehat\bP^\2}{\bR^\2},
\end{equation}
where, in the notation above, the two-endpoint marginal is applied after the pushforward by $\Phi$. The measure $\widehat\bP^\2_{0T}$, called the \textit{static Schr\"odinger bridge}, solves the \textit{static SBP}:
\begin{equation*}%\label{eqn: static sbp (comparison)}
    \widehat\bP^\2_{0T} := \arg\min\set{\KL{\bP_{0T}}{\bR^\2_{0T}} \text{ such that }  \bP_0=\rho_0,\bP_T=\rho_T}.
\end{equation*}
\cref{Thm: KL divergence comparison} establishes conditions under which equality is achieved everywhere throughout \eqref{eqn: KL two-step inequality}.
On the other hand, \eqref{eqn: KL two-step inequality} indicates that, if $\Phi_{\#}\widehat\bP^\1 \neq \widehat\bP^\2$, then at least one of the following is true: 1) the measure $\Phi_{\#}\widehat\bP^\1_{0T}$ is not the optimizer of the static SBP for $\bR^\2$, or 2) the ``bridge'' component of $\Phi_{\#}\widehat\bP^\1$, i.e., the conditional distribution of $\Phi_{\#}\widehat\bP^\1$ conditioning on the endpoints $(\omega(0),\omega(T))$, does not match that of $\bR^\2$.\\

\subsection{The Case of Deterministic Initial Data} Let us now discuss an important special case: the Dirac delta initial data, or, in other words, the \textit{deterministic} initial data. In fact, when $\bR_0=\delta_{x_0}$ for some $x_0\in\R^d$, the static Schr\"odinger system \eqref{Equation: fg Schrodinger System} always admits the unique (up to rescaling) solution $\f\equiv 1,\g=d\rho_T/d\bR_T$, and thus \cref{Thm: KL divergence comparison} takes a simpler form.
\begin{cor}
    Let $\bR^\1$ be a reference measure with $\bR_0^\1=\delta_{x_0}$ for some $x_0\in\R^d$, and suppose $\rho^\1_0=\bR^\1_0$ and $\KL{\rho_T^\1}{\bR_T^\1}<\infty$. Take $\bR^\2,\rho^\2_0,\rho^\2_T$ to be of the same relation to $\bR^\1,\rho^\1_0,\rho^\1_T$ as in the hypothesis of \cref{Thm: KL divergence comparison}. Then, the Schr\"odinger bridges $\widehat\bP^\1$ and $\widehat\bP^\2$ exist and take the form $$\widehat\bP^{(k)}=\g^{(k)}(\db{X_T,\t_T},\Lambda_T)\bR^{(k)}\textit{ for }k=1,2,$$
    and the following are equivalent:
    \begin{enumerate}[label=\textup{(\roman*)}]
        \item $\KL{\widehat\bP^\1}{\bR^\1} = \KL{\widehat\bP^\2}{\bR^\2}$.
        \item $\g^\1=\g^\2$.
        \item For $\bR^\1_T$-a.e. $\db{x,\t}$, the target distribution in the regime $\d$ satisfies
        \[
        \frac{d\rho^\1_T}{d\bR^\1_T}(\db{x,\t},\d)=\frac{d\rho_T^\2}{d\bR^\2_T}(\db{\c,\t},\d).
        \] 
    \end{enumerate}
\end{cor}
\noindent \textbf{Remark.} As a continuation of the remark above, we note that in the case $\bR^\1_0=\delta_{x_0}$, it is \textit{always} true that $\Phi_{\#}\widehat\bP^\1_{0T}=\widehat\bP^\2_{0T}$. To see this, set $\f^\1\equiv 1$ and $\g^\1 \equiv d\rho^\1_T /d\bR^\1_T$ and recall from \eqref{eqn:cross-regime density} that $
\bR^\1_T(\db{x,\t},\d)=\tilde q(0,x_0,x,\t)=V(x,\t)q(0,x_0,x,\t)$.
Applying \eqref{eqn: pushforward density}, we have that the Radon-Nikodym derivative  
$\,d\Phi_{\#}\widehat\bP^\1_{0T}\,/d\bR^\2_{0T}\,$ at the point 
$\left(\begin{array}{c} 
    (\aug{x_0},\a) \\ 
    (\aug{y},\a)
    \end{array} \right)$
is simply $\g^\1(\aug{y},\a)$, and thus equal to $\g^\2(\aug{y},\a)$. Further, 
\begin{equation*}%\label{eqn:f20T RN derivative}
\begin{aligned}
\frac{d\Phi_{\#}\widehat\bP^\1_{0T}}{d\bR^\2_{0T}}\left(\begin{array}{c} 
    (\aug{x_0},\a) \\ 
    (\db{\c,\t},\d)
    \end{array} \right)&=\bE_{\bR^\1_{0T}}\brac{\g^\1(\db{X_T,\t_T},\Lambda_T)\bigg|\t_T=\t,\Lambda_T=\d}\\
&=\frac{\int_{\R^d}q(0,x_0,\t,x)V(\t,x)\g^\1(\db{x,\t},\d)dx}{\int_{\R^d}q(0,x_0,\t,x)V(\t,x)dx}\\
&=\frac{\int_{\R^d}\rho_T^\1(\db{x,\t},\d)dx}{\int_{\R^d}q(0,x_0,\t,x)V(\t,x)dx}\;\textup{(because $\g^\1$ satisfies \eqref{eqn: Scenario (1) third equation Schrodinger system})}\\
&=\frac{\rho_T^\2(\db{\c,\t},\d)}{\bR_T^\2(\db{\c,\t},\d)}=\g^\2(\db{\c,\t},\d).
\end{aligned}
\end{equation*}
Thus, we have $\,d\Phi_{\#}\widehat\bP^\1_{0T}\,/d\bR^\2_{0T}=\f^\2\g^\2$, which implies $\Phi_{\#}\widehat\bP^\1_{0T}=\widehat\bP^\2_{0T}$.\\

\noindent We are now ready to answer the questions posed in the introduction of the section:\vspace{0.3cm}

\noindent (1) First, Observer 1, using the uSBP model from Scenario (1), cannot obtain a smaller KL divergence than Observer 2 using the model from Scenario (2).
This is because the diffusion with killing under $\widehat\bP^\2$ produces a spatial distribution $\rho^*_T(\db{x,\t},\d)$, which is the (unobserved) location distribution of the killed particles that minimizes $\KL{\bP}{\bR^\1}$ among all the distributions such that $\int_{\R^d}\rho_T(\db{x,\t},\d)dx=\rho^\2_T(\db{\c,\t},\d)$. In general, this $\rho^*_T$ needs not to be equal to the distribution $\rho^\1_T(\db{x,\t},\d)$ seen by Observer 1, unless (very unlikely!) it happens that \eqref{eqn:rho1 for equality} holds.\vspace{0.3cm}

\noindent (2) Likewise, Observer 2 should not adopt the uSBP framework from Scenario (1) with an arbitrary choice of target spatial distribution. Not only will it not match the true underlying $\rho^\1_T$, but it will also yield a greater KL divergence than the less restrictive approach of Scenario (2). The only exception is the special case of the deterministic initial condition. In this case, Observer 2 can compute $\rho^\1_T$ using the corollary above and then follow the uSBP setup from Scenario (1).

\subsection{Comparison Among Other Scenarios} By appropriately defining $\Phi$ between trajectory spaces, the results obtained in \cref{Thm: KL divergence comparison} can be extended to establish comparisons between any two uSBP scenarios. 

For the four Scenarios considered in this paper, we summarize the relations in the following theorem.
\begin{thm}
    Suppose that the SDE coefficients $b,\sigma$, the killing rate $V$, the initial data $\bR_0$, and the target distributions $\rho_0(\cdot,\a),\rho_T(\cdot,\a)$ are identical for uSBP Scenarios (1)-(4), and that the target distributions $\rho_T(\cdot,\d)$ satisfy for all $x\in\R^d, \t \in [0,T]$:
    \begin{equation*}
    \begin{dcases}
        \int_{\R^d}\rho^\1_T(\db{x,\t},\d)dx = \rho^\2_T (\db{\c,\t},\d), \\
        \int_0^T\rho^\1_T(\db{x,\t},\d)d\t = \rho^\3_T (x,\d) ,\\
        \int_0^T\rho^\2_T (\db{\c,\t},\d) d\t= \rho^\4_T(\c,\d),\\
        \int_{\R^d}\rho^\3_T (x,\d) dx= \rho^\4_T(\c,\d).
    \end{dcases}
    \end{equation*}
    Then, the following relations hold among uSBP Scenarios (1)-(4):
    \begin{equation*}
        \begin{aligned}
            \KL{\widehat\bP^\1}{\bR^\1} \geq \KL{\widehat\bP^\2}{\bR^\2}, \\
            \KL{\widehat\bP^\1}{\bR^\1} \geq \KL{\widehat\bP^\3}{\bR^\3}, \\
            \KL{\widehat\bP^\2}{\bR^\2} \geq \KL{\widehat\bP^\4}{\bR^\4}, \\
            \KL{\widehat\bP^\3}{\bR^\3} \geq \KL{\widehat\bP^\4}{\bR^\4},
        \end{aligned}
    \end{equation*}
    and
    \begin{equation*}
        \begin{aligned}
            \KL{\widehat\bP^\1}{\bR^\1} &= \KL{\widehat\bP^\2}{\bR^\2} \iff \g^\1(\db{x,\t},\d) \text{ does not depend on $x$},\\
            \KL{\widehat\bP^\1}{\bR^\1} &= \KL{\widehat\bP^\3}{\bR^\3} \iff \g^\1(\db{x,\t},\d) \text{ does not depend on $\t$},\\
            \KL{\widehat\bP^\2}{\bR^\2} &= \KL{\widehat\bP^\4}{\bR^\4} \iff \g^\2(\db{\c,\t},\d) \text{ is a constant function},\\
            \KL{\widehat\bP^\3}{\bR^\3} &= \KL{\widehat\bP^\4}{\bR^\4} \iff \g^\3(x,\d) \text{ is a constant function}.
        \end{aligned}
    \end{equation*}
\end{thm}
In each of these cases, there exists a unique target distribution $\rho_T(\cdot,\d)$ that achieves the equality. As evidenced by the comparison between Scenarios (1) and (2), it is determined by an equation of the form \eqref{eqn:rho1 for equality}, which requires solving another SBP.
\section{Connection to Previous uSBP Works}\label{Section: connection to other paper}
\subsection{Removing the Constraint on Surviving Paths}
There are some uSBP problems that require modifications of the regime-switching approach proposed in this paper. Consider for example the uSBP in \cite{eldesoukey2024excursion}: as in our Scenario (1), the observer knows the initial distribution $\rho_0=\rho_0(\cdot,\a)$ and the spatio-temporal distribution of killed particles $\rho_T(\db{x,\t},\d)$, but \textbf{not} the spatial distribution of surviving particles $\rho_T(\cdot,\a)$. Consequently, the optimization \eqref{dfn: SBP-dynamic} is done over path measures such that $\bP_T(\cdot,\d)=\rho_T(\cdot,\d)$, but with no constraint on $\bP_T(\cdot,\a)$. We call this problem Scenario ($\star$) and define it as follows.
\begin{dfn}[Scenario ($\star$)]
    With $\bR,\rho_0,\rho_T$ as in \sce{1}, determine
    \[
    \widehat{\bP}^\star:=\arg\min\set{\KL{\bP}{\bR} \text{ such that } \bP_0=\rho_0,\bP_T(\cdot,\d)=\rho_T(\cdot,\d) }.
    \]    
\end{dfn}
  
At first glance, the regime-switching approach cannot handle this problem, because our setup assumes that the target distribution $\rho_T$ prescribes the \textit{entire} joint distribution over $\R^{d+1}\times\set{\a,\d}$. On the other hand, the results that we have obtained in our Scenario (1) match those of \cite{eldesoukey2024excursion}, with the only difference being that $\g(\cdot,\a)\equiv 1$ in \cite{eldesoukey2024excursion}, meaning that, for example,
\[
\varphi(t,x,\a)=\int_{\R^d}q(t,x,T,y)dy + \int_{\R^d}\int_t^T q(t,x,r,y)V(r,y)\g(\psi(r,y),\d)dr dy.
\]
Notably, the relation \eqref{eq: marginal density killed regime} and the dynamic Schr\"odinger system in \cref{Dynamic Schrodinger System} take the same form in our Scenario (1) and in \cite{eldesoukey2024excursion}.

Let us elucidate the connection between Scenario ($\star$) and \sce{1} formally. We start with the following lemma that relates $\widehat\bP^\star$ to $\widehat\bP^\1$.
\begin{lem}
     Suppose that the uSBP in Scenario $(\star)$ with the reference measure $\bR$ and the target distributions $\rho_0,\rho_T(\cdot,\d)$ (where $\rho_0$ is a probability measure and $\rho_T(\cdot,\d)$ is a sub-probability measure on $\R^d$) admits a unique solution $\widehat\bP^\star$. Then there exists a unique (sub-probability) measure $\rho_T(\cdot,\a)$ on $\R^d$ such that $\rho_T:=(\rho_T(\cdot,\a),\rho_T(\cdot,\d))$ is a probability measure on $\R^d\times\set{\a,\d}$, and $\widehat\bP^\star=\widehat\bP^\1$ where $\widehat\bP^\1$ is the Schr\"odinger bridge obtained in Scenario (1) with the reference measure $\bR$ and the target distributions $\rho_0,\rho_T$.
\end{lem}
\noindent The conclusion of the lemma follows from the simple observation that if $\rho_T(\cdot,\a):=\widehat\bP^\star_T(\cdot,\a)$, then the optimality in both scenarios ensures that $\widehat\bP^\1 = \widehat\bP^\star$, and this choice of $\rho_T(\cdot,\a)$ is obviously unique. Moreover, this lemma implies that $\widehat\bP^\star$ will inherit the properties of $\widehat\bP^\1$ as described in \cref{section: general results} and \cref{section: case-by-case results}. 

It still remains to explain why $\g(\cdot,\a)$ is a constant function for $\widehat\bP^\star$. We will demonstrate in the theorem below that this indeed is not a coincidence with the help of a properly defined pushforward of $\bR$.

\begin{thm}\label{thm: comparison to other paper}
    Let $\widehat\bP^\star$ be the Schr\"odinger bridge in Scenario $(\star)$ with the reference measure $\bR$ and the target distributions $\rho_0,\rho_T(\cdot,\d)$. 
    Then, there exist non-negative measurable functions $\f,\g$ such that $\widehat\bP^\star=\f(\db{X_0,\t_0},\Lambda_0)\g(\db{X_T,\t_T},\Lambda_T)\bR$ and $\g(\cdot,\a)\equiv 1$.
\end{thm}
\begin{proof}
    Let $\widehat\bP^\1$ be the Schr\"odinger bridge from Scenario (1) with the reference measure $\bR$ and the target distributions $\rho_0$ and $\rho_T=(\widehat\bP^\star_T(\cdot,\a),\rho_T(\cdot,\d))$. As stated in the lemma above, $\widehat\bP^\1=\widehat\bP^\star$. We also know that $\widehat\bP^\1=\f^\1(\db{X_0,\t_0},\Lambda_0)\g^\1(\db{X_T,\t_T},\Lambda_T)\bR$, where $(\f^\1,\g^\1)$ is the unique solution of the Schr\"odinger system in Scenario (1). We will show that $\g^\1(\cdot,\a)$ is a constant function, and the statement of the theorem will then follow. 
    
    We begin by defining an appropriate trajectory space. 
    Let $\Omega^\dagger$ be the set of c\`adl\`ag paths $t\in[0,T]\mapsto\omega(t)=(\db{X_t,\t_t},\Lambda_t)$ such that:
    \begin{itemize}%\label{dfn:omega-dagger}
    \item either, $\Lambda_t=\a$ for all $t\in [0,T]$, in which case $\db{X_t,\t_t}$ is continuous on $[0,T)$, and $\db{X_T,\t_T}=\db{\c,0}$. 
    \item or, there exists a time $\kt=\kt(\omega)\in[0,T]$ such that $\Lambda_t=\a$  for $t\in[0,\kt)$ and $\Lambda_t=\d$ for $t\in[\kt,T]$, in which case $\db{X_t,\t_t}$ is continuous on $[0,\kt)$ and $\db{X_t,\t_t}=\db{X_{\kt-},\kt}$ for all $t\in[\kt,T]$.
\end{itemize}
Similarly as in $\Omega^\1$, the killed paths in $\Omega^\dagger$ are also frozen in ``time and space'' upon killing, but different from $\Omega^\1$, the paths that survived in $\Omega^\dagger$ are all collapsed into a singular point $\c$ at the terminal time $T$. This naturally leads to the mapping  $\Phi^\dagger:\Omega^\1 \to \Omega^\dagger$ as follows:
\begin{equation*}
\Phi^\dagger(\omega)(t):=
\begin{dcases}
(\aug{\c},\a) & \text {if }t=T,\Lambda_T=\a,   \\  
\omega(t)   & \text {otherwise. } 
\end{dcases} 
\end{equation*}
This map clearly eliminates the constraint on the terminal distribution of surviving particles, while still being able to almost surely reconstruct the original paths. Note that just like $\Phi$ in the previous section, $\Phi^\dagger$ is a bijection, and its inverse is given by:
\begin{equation*}
(\Phi^\dagger)^{-1}(\omega)(t):=
\begin{dcases}
\omega(t-) & \text {if }t=T,\Lambda_T=\a,   \\  
\omega(t)   & \text {otherwise. } 
\end{dcases} 
\end{equation*}

Define $\bR^\dagger:=\Phi^\dagger_{\#}\bR$ and $\rho^\dagger_0:=\rho_0$, $\rho^\dagger_T(\cdot,\d):=\rho_T(\cdot,\d)$ and let $\rho^\dagger_T(\cdot,\a)$ be the Dirac delta distribution at $\aug{\c}$ such that $\int_0^T\int_{\R^d}\rho_T(\db{x,\t},\d)dxd\t+\rho^\dagger_T(\aug{\c},\a)=1$. Consider the SBP with the reference measure $\bR^\dagger$ and the target distributions $\rho_0^\dagger,\rho_T^\dagger$, called SBP-$\dagger$. Here, $\bR^\dagger$ is not the path measure of a regime-switching diffusion, but the general SBP theory still applies. In particular, the solution $\widehat\bP^\dagger$ satisfies
\begin{equation}\label{eqn:fg format of P^dagger}
\widehat\bP^\dagger=\f^\dagger(\db{X_0,\t_0},\Lambda_0)\g^\dagger(\db{X_T,\t_T},\Lambda_T)\bR^\dagger,
\end{equation}
where $\f^\dagger,\g^\dagger$ are non-negative measurable functions that solve the Schr\"odinger system \eqref{Equation: fg Schrodinger System} for $(\bR^\dagger,\rho_0^\dagger,\rho_T^\dagger)$.
Clearly, if $\bP$ is admissible in Scenario ($\star$), i.e., $\bP_0=\rho_0$ and $\bP_T(\cdot,\d)=\rho_T(\cdot,\d)$, then $\bP^\dagger:=\Phi^\dagger_{\#}\bP$ is admissible for SBP-$\dagger$. Further, since $\Phi^\dagger$ is a bijection, the same argument as in \cref{lem:KL of pushforward by Phi} leads to $\KL{\bP^\dagger} {\bR^\dagger}=\KL{\bP}{\bR}$. Likewise, if $\bP^\dagger$ is a path measure on $\Omega^\dagger$ that is admissible for SBP-$\dagger$, then $\bP:=(\Phi^\dagger)^{-1}_{\#}\bP^\dagger$ is admissible for Scenario ($\star$), and $\KL{\bP}{\bR}=\KL{\bP^\dagger}{\bR^\dagger}$. As a result, it must be that
\[
\widehat\bP^\dagger=\Phi^\dagger_{\#}\widehat\bP^\star=\Phi^\dagger_{\#}\widehat\bP^\1.
\]

    By the same argument as in the proof of \cref{Thm: KL divergence comparison}, applying \eqref{eqn: pushforward density}, if $z\in\Omega^\dagger$ and $\tilde{\omega}:=(\Phi^\dagger)^{-1}(z)\in\Omega^\1$, then
    \[
    \frac{d\widehat\bP^\dagger}{d\bR^\dagger}(z)
    =\bE_{\bR}\brac{\frac{d\widehat\bP^\1}{d\bR} \bigg| \Phi^\dagger=z}=\frac{d\widehat\bP^\1}{d\bR}(\tilde{\omega})=\f^\1(\tilde{\omega}(0))\g^\1(\tilde{\omega}(T)).
    \]
Combining with \eqref{eqn:fg format of P^dagger}, we have that \[
\f^\dagger(z(0))\g^\dagger(z(T))=\f^\1(\tilde{\omega}(0))\g^\1(\tilde{\omega}(T)).
\]
To obtain information on $\g^\1(\cdot,\a)$, we are interested in the Radon-Nikodym derivative for paths $z$ that end in the active regime $\a$. Observe that if $\Lambda_T=\a$, then $\tilde{\omega}(0)=z(0)=(\aug{x},\a)$ and $\tilde{\omega}(T)=(\aug{y},\a)$ for some $x,y\in\R^d$. Hence for Lebesgue-a.e. $x,y$,
\[
\f^\dagger(\aug{x},\a)\g^\dagger(\aug{\c},\a)=\f^\1(\aug{x},\a)\g^\1(\aug{y},\a),
\]
and thus $\g^\1(\cdot,\a)$ is a positive constant, which we denote by $c$.

To complete the proof, define $\f:=c\,\f^\1$ and $\g:=\g^\1/c$. Then, we have in Scenario ($\star$), $\widehat\bP^\star=\f(\db{X_0,\t_0},\Lambda_0)\g(\db{X_T,\t_T},\Lambda_T)\bR$, where $\g(\cdot,\a)\equiv 1$.
\end{proof}
We have shown that removing the constraint $\bP_T(\cdot,\a)=\rho_T(\cdot,\a)$ in Scenario (1) preserves the overall structure of the unbalanced Schr\"odinger bridge, with $\g(\cdot,\a)\equiv 1$.
In another related uSBP work \cite{eldesoukey2025inferring}, similarly to Scenario ($\star$), the authors considered uSBPs as in our Scenarios (2) and (3), but with no constraint $\rho_T(\cdot,\a)$. As in Scenario ($\star$), it is easy to see that these unbalanced Schr\"odinger bridges inherit properties of $\widehat\bP$ from Scenarios (2) and (3) respectively, and repeating the argument in \cref{thm: comparison to other paper} will lead to $\g(\cdot,\a)\equiv 1$.
\subsection{Further Strengths of the Regime-Switching Approach to the uSBP}%\label{Subsection: uSBP discussion}
To conclude the analysis of the uSBP, let us discuss further advantages of the regime-switching approach and go beyond the four scenarios studied above. 

\subsubsection{Incomplete observational data}
First, as an easy extension, consider the situation where the spatial observation is ``foggy'', in the sense that the observer can see the killing location only if the killing time is in some Borel $E\subseteq [0,T]$. For example, suppose that the observer cannot see where the particle was killed if the event occurred too early, say at $t\in[0,T/2]$, or a physical obstacle is blocking the observation for $t$ in some interval $[t_1,t_2]$. To model this uSBP scenario, it suffices to set $\psi(t,x)=\one_E(t) \db{x,t} + \one_{E^c}(t) \db{\c,t}$, and the general results of \cref{section: general results} apply.

\subsubsection{Two marginal distributions in the dead regime}
Next, consider a mix of Scenarios (2) and (3) in the following way: at every $t\in[0,T]$, the observer knows the amount of mass lost, but not its spatial distribution; further, at time $T$, the observer has access to the overall (cumulative) spatial distribution of killed particles. In other words, if $\rho_T(\db{x,\t},\d)$ is the full spatio-temporal distribution of killed particles, the observation provides the two marginals $\int_0^T \rho_T(\db{x,\t},\d) d\t$ and $\int_{\R^d} \rho_T(\db{x,\t},\d) dx$.  In the spirit of the SBP, the objective is to match these observations while minimizing KL divergence to a reference measure. As in the case with Scenario $(\star)$, our regime-switching approach cannot model this behavior directly, and some modifications are required. This is, once again, a ``SBP within a SBP'': for every choice of $\rho_T(\db{x,\t},\d)$ in Scenario (1) that satisfies these marginal constraints, we obtain a corresponding $\widehat\bP^\1$, and then we must choose the one with smallest KL divergence to $\bR^\1$.
However, another possibility in modeling this problem is to use three regimes $\a,\d_1,\d_2$, with active particles switching to regimes $\d_1$ and $\d_2$ at a rate $V/2$ each. A particle switching to $\d_1$ would be transported to $\psi_1 = \db{x,0}$ and a particle switching to $\d_2$ would be transported to $\psi_2=\db{\c,\t}$. The target distribution $\rho_T(\cdot,\d_1)$ would be the observed distribution of the killing locations, and the target distribution $\rho_T(\cdot,\d_2)$ would be the observed distribution of the killing times. Then, one could restrict the set of admissible measures $\bP$ to be those such that, in addition to satisfying the target distributions $(\rho_0,\rho_T)$, the switching rates $\a \to \d_1$ and $\a \to \d_2$ under $\bP$ would be equal. If using the stochastic control formulation, with $\xi_1,\xi_2$ being the controls of the switching rates $\a \to \d_1$ and $\a \to \d_2$ respectively, one would impose the additional constraint that $\xi_1=\xi_2$.

\subsubsection{Death and revival}
Finally, suppose that ``revival'' is possible, that is, 
\[
\lim_{\eps \downarrow 0}\frac{1}{\eps}\bR(\Lambda_{t+\eps}= \a | \Lambda_t = \d, X_t=x) = Q(t,x),
\]
where $Q\geq0$ is the revival rate. In addition, a particle that switches regimes $\d\to\a$ at time $t$ jumps to a new location $\tilde{\psi}(t,X_{t-})$.

Here, a broad range of observation scenarios can be modeled by using the dimension augmentation technique and varying $\psi,\tilde{\psi}$ appropriately. To give some examples, by sufficiently augmenting the dimension, the additional ``information tracking'' coordinates can record:
\begin{itemize}
    \item the number of regime-switches $i\to j$,
    \item the position and/or the time of the first/last $n$ regime-switches $i\to j$,
    \item the mean time between regime switches,
    \item the number of deaths that occurred in a specific $B\subseteq \R^d$, etc.
\end{itemize}
This flexibility is the result of a combination of two factors. First, although the dynamics of the augmented coordinates are only indirectly driven by the regime-switching diffusion, they are still part of its spatial dimension. Second, the general model of the SBP for regime-switching (jump) diffusions in \cite{zlotchevski2025schrodinger} allows $\psi,\tilde{\psi}$ to depend on the spatial variable including the augmented dimensions. Therefore, the augmented coordinates not only can ``read'' the information contained in the current state of the process $(t,X_t)$, but can also access the information contained in all the augmented coordinates themselves. For example, if we would like the augmented coordinate $\mathfrak{n}$ to count the number of times that a particle has been killed, then we can initialize $\mathfrak{n}=0$ at time $t=0$ and set $\psi(t,\db{x,\mathfrak{n}})=\db{x,\mathfrak{n}+1}$. If we also want to count the number of revivals, then we can add a second augmented coordinate $\mathfrak{m}$, set $\mathfrak{m}=0$ at time $t=0$, and choose $\psi(t,\db{x,\mathfrak{n},\mathfrak{m}})=\db{x,\mathfrak{n}+1,\mathfrak{m}}$ and $\tilde{\psi}(t,\db{x,\mathfrak{n},\mathfrak{m}})=\db{x,\mathfrak{n},\mathfrak{m}+1}$. By doing so, we can impose additional constraints in the uSBP through the target distribution $\rho_T$ in the augmented dimensions.

\bibliographystyle{plain} 
\bibliography{references}

\end{document}